\documentclass[10pt,english]{amsart}
\usepackage{amsfonts, amssymb, amsmath, amsthm, eucal, latexsym}

\usepackage[english]{babel}

\usepackage{ams fonts}
\usepackage{amsmath}
\usepackage{amsthm}
\usepackage{amssymb,plain}
\usepackage{graphics}
\usepackage[all]{xy}
\usepackage[latin1]{inputenc}
\usepackage{mathrsfs}

\def\<{\langle}
\def\>{\rangle}

\def\k{\kappa}

\def\n{{\bf n}}

\def\Chi{\raise .3ex \hbox{\large $\chi$}} 

\def\k{\Bbbk}

\def\n{\mathrm{n}}
\def\rk{\mathrm{rk}}
\def\A{\mathrm{A}}
\def\B{\mathrm{B}}
\def\C{\mathrm{C}}
\def\D{\mathrm{D}}
\def\E{\mathrm{E}}
\def\F{\mathrm{F}}
\def\G{\mathrm{G}}

\def\[{\Bigl [}
\def\]{\Bigr ]}
\def\({\Bigl (}
\def\){\Bigr )}
\def\[{\Bigl [}
\def\]{\Bigr ]}
\def\({\Bigl (}
\def\){\Bigr )}

\newtheorem{theo}{THEOREM}[section]
\newtheorem{Lemme}{Lemma}[section]
\newtheorem{prop}{Proposition}[section]
\newtheorem{rem}{Remark}[section]

\newtheorem{cor}{Corollary}[section]

\title[On the nullcone]
{On related varieties to the commuting variety of a semisimple Lie algebra}

\author{Mouchira Zaiter}

\address{Universit\'e Paris 7 - CNRS \\
Institut de Math\'ematiques de Jussieu \\
Th\'eorie des groupes \\
Case 7012 \\ B\^atiment Chevaleret \\
75205 Paris Cedex 13, France}

\email{zaiter@math.jussieu.fr}

\begin{document}

\date\today
\maketitle


\bibliographystyle{plain}

\begin{abstract}  Let $\mathfrak{g}$ be a semisimple Lie algebra of finite dimension. The nullcone $\mathcal{N}$ of $\mathfrak{g}$ is the set of $(x,y)$ in $\mathfrak{g}\times\mathfrak{g}$ such that $x$ and $y$ are nilpotents and are in the same Borel subalgebra. The main result of this paper is that $\mathcal{N}$ is a closed and irreducible subvariety of $\mathfrak{g}\times\mathfrak{g}$, its normalization has rational singularities and its normalization morphism is bijective.
\end{abstract}

\tableofcontents

\section{Introduction}
 
The basic field $\Bbbk$ is an algebraic closed field of charecteristic zero. Let $\mathfrak{g}$ be a semisimple Lie algebra of finite dimension and let $G$ be its adjoint group. We denote by $\mathfrak{b}$ a Borel subalgebra of $\mathfrak{g}$, $\mathfrak{h}$ a Cartan subalgebra of $\mathfrak{g}$ contained in $\mathfrak{b}$, $b_\mathfrak{g}$ and $\rk\mathfrak{g}$ the dimensions of $\mathfrak{b}$ and $\mathfrak{h}$ respectively. Let $W$ be the Weyl group of $\mathfrak{g}$ with respect to $\mathfrak{h}$. The symmetric algebra of $\mathfrak{g}$ is denoted by $S(\mathfrak{g})$ and the subalgebra of its $G$-invariant elements is denoted by $S(\mathfrak{g})^G$. Let $\mathcal{B}_\mathfrak{g}$ be the set of $(x,y)$ in $\mathfrak{g}\times\mathfrak{g}$ such that $x$ and $y$ are in the same Borel subalgebra. In Section 4, we show the following theorem:
\begin{theo}
\begin{enumerate} 
 \item[(i)] The variety $\mathcal{B}_\mathfrak{g}$ is closed and irreducible of dimension $3b_\mathfrak{g}-\mathrm{rk}\mathfrak{g}$, but it isn't normal,
 \item[(ii)] the algebra of $W$-invariant regular functions on $\mathfrak{h}\times\mathfrak{h}$ is isomorphic to the algebra of $G$-invariant regular functions on $\mathcal{B}_\mathfrak{g}$.
\end{enumerate}
\end{theo}
The variety $\mathcal{B}_\mathfrak{g}$ contains two interesting varieties, the nullcone and the commuting variety. The nullcone of $\mathfrak{g}$, denoted by $\mathcal{N}$, is defined by the set of zeros of the $G$-invariant polynomial functions $f$ on $\mathfrak{g}\times\mathfrak{g}$ such that $f(0)=0$. By \cite{13}, it can be equivalently defined as the variety of pairs $(x,y)$ in $\mathcal{B}_\mathfrak{g}$ such that $x$ and $y$ are nilpotents. It is well-know that the nullcone plays a fundamental role in the theory of invariants and its applications. This cone was introduced and studied by D. Hilbert in his famous paper ''Ueber die vollen Invariantensystem''(\cite{25}). H. Kraft and N. R. Wallach studied in \cite{22} the geometry of the nullcone; they showed that the nullcone of any number of copies of the adjoint representation of $G$ on $\mathfrak{g}$ is irreducible and they gave a resolution of singularities of this variety. Moreover, in \cite{26}, they studied when the polarizations of a set of invariant functions defining the nullcone of a representation $V$ define the nullcone of a direct sum of several copies of $V$. This question was earlier studied by M. Losik, P. W. Michor, and V. L. Popov in \cite{29}. In \cite{27}, V. L. Popov gives a general algorithm to determine the irreducible components of maximal dimension of the nullcone using the weights of the representation and their multiplicities. In Section 5 of this paper, we show the following theorem:
\begin{theo}
\begin{enumerate}
 \item[(i)] The nullcone $\mathcal{N}$ is closed and irreducible of dimension $3(b_\mathfrak{g}-\mathrm{rk}\mathfrak{g})$, 
 \item[(ii)] the codimension of the set of its singular points is at least four and the normalization morphism of $\mathcal{N}$ is bijective.
\end{enumerate}
\end{theo}
In the last section of this paper, we use the ideas of Wim H. Hesselink (see \cite{11}) for showing the following theorem:
\begin{theo}\label{Tnrs} The normalizations of $\mathcal{N}$ and $\mathcal{B}_\mathfrak{g}$ have rational singularities.
\end{theo}
The other intersting variety of $\mathcal{B}_\mathfrak{g}$ is the commuting variety. We recall that the commuting variety $\mathcal{C}_\mathfrak{g}$ of $\mathfrak{g}$ is the set of $(x,y)$ in $\mathfrak{g}\times\mathfrak{g}$ such that $[x,y]=0$. We used a result of A. Josef in \cite{23} on $\mathcal{C}_\mathfrak{g}$ to prove that $\k[\mathcal{B}_\mathfrak{g}]^G$ (see Notations) and $S(\mathfrak{h}\times\mathfrak{h})^W$ are isomorphic.

In this paper, we also proved some other results on $\mathcal{B}_\mathfrak{g}$. We used the homogenous generators $p_1,\ldots,p_{\mathrm{rk}\mathfrak{g}}$ of $S(\mathfrak{g})^G$, chosen so that the sequence of their degrees $d_1,\ldots, d_{\mathrm{rk}\mathfrak{g}}$ is increasing. By \cite{3}, $d_1+\ldots+d_{\rk\mathfrak{g}}=b_{\rk\mathfrak{g}}$. Now, let $X:=\sigma(\mathfrak{h}\times\mathfrak{h})$, whenever $\sigma$ is the morphism from $\mathfrak{g}\times\mathfrak{g}$ to $\k^{b_{\mathfrak{g}}+\mathrm{rk}\mathfrak{g}}$ defined by $$\sigma(x,y)=(p_1^{(0)}(x,y),\ldots,p_1^{(d_1)}(x,y),\ldots,p_{\mathrm{rk}\mathfrak{g}}^{(0)}(x,y),\ldots,p_{\mathrm{rk}\mathfrak{g}}^{(d_{\mathrm{rk}\mathfrak{g}})}(x,y)),$$ with $p_i^{(n)}$ are the $2$-order polarizations of $p_i$ of bidegree $(d_i-n,n)$ (see Notations). It is shown in Section 3 that there is an isomorphism between $(\mathfrak{h}\times\mathfrak{h})/W$ and the normalization of $X$, the normalization morphism of $X$ is bijective, and the codimension of the complement of the set of regular points of $X$ is at least two. On the other hand, the variety $\sigma(\mathcal{B}_\mathfrak{g})$ is equal to $X$ and $\mathcal{B}_\mathfrak{g}$ is an irreducible component of $\sigma^{-1}(X)$.

Finally, in the appendix, we give some results on the positive roots related to Theorem \ref{Tnrs}.

\section{Notations}

We consider the diagonal action of $G$ on $\mathfrak{g}\times\mathfrak{g}$ and the diagonal action of $W$ on $\mathfrak{h}\times\mathfrak{h}$. Let $\mathfrak{u}$ be the set of the nilpotent elements in $\mathfrak{b}$, let $\mathcal{R}$ be the root system of $\mathfrak{h}$ in $\mathfrak{g}$, let $\mathcal{R}_+$ be the positive root system defined by $\mathfrak{b}$ and let $\Pi$ be the set of simple roots of $\mathcal{R}_+$. We denote by $\mathbf{B}$ the normalizer of $\mathfrak{b}$ in $G$, $\mathbf{U}$ its unipotent radical, $\mathbf{H}$ and $N_G(\mathfrak{h})$ the centralizer and the normalizer of $\mathfrak{h}$ in $G$. We use the following notations:
\begin{itemize}
\item if $G\times A\longrightarrow A$ is an action of $G$ on the algebra $A$, denote by $A^G$ the subalgebra of the $G$-invariant elements of $A$,
\item for $i=1,\ldots,\rk\mathfrak{g}$, the 2-order polarizations of $p_i$ of bidegree $(d_i-n,n)$ denoted by $p_i^{(n)}$ are the unique elements in $(S(\mathfrak{g})\otimes_\C S(\mathfrak{g}))^G$ satisfying the following relation $$p_i(ax+by)=\sum\limits_{n=0}^{d_i}a^{d_i-n}b^n p_i^{(n)}(x,y),$$ for all $a,b\in\mathbb{C}$ and $(x,y)\in \mathfrak{g}\times\mathfrak{g}$,
\item for $i=1,\ldots,\rk\mathfrak{g}$, $\varepsilon_i$ is the element of $S(\mathfrak{g})\otimes_\mathbb{C}\mathfrak{g}$ defined by $$\<\varepsilon_i(x),v\>=p_i'(x)(v),\mbox{ }\forall x,v\in\mathfrak{g},$$ whenever $p_i'(x)$ is the differential of $p_i$ at $x$ for $i=1,\ldots\rk\mathfrak{g}$, 
\item for $i=1,\ldots,\rk\mathfrak{g}$, the $2$-polarizations of $\varepsilon_i$ of bidegree $(d_i-m-1,m)$ denoted by $\varepsilon_i^{(m)}$ are the unique elements in $S(\mathfrak{g})\otimes_\mathbb{C}S(\mathfrak{g})\otimes_\mathbb{C}\mathfrak{g}$ satisfying the following relation $$\varepsilon_i(ax+by)=\sum\limits_{n=0}^{d_i-1}a^{d_i-n-1}b^n \varepsilon_i^{(n)}(x,y),$$ for all $a,b\in\mathbb{C}$ and $(x,y)\in \mathfrak{g}\times\mathfrak{g}$,
\item $\mathfrak{g}'$ is the set of regular elements of $\mathfrak{g}$,
\item $\mathfrak{h}'$ is the subset of regular elements of $\mathfrak{g}$ belonging to $\mathfrak{h}$, 
\item for $x$, $y$ in $\mathfrak{g}$, $P_{x,y}$ is the subspace of $\mathfrak{g}$ generated by $x$ and $y$,
 \item $\pi$ is the morphism from $\mathfrak{h}$ to $\k^{\mathrm{rk}\mathfrak{g}}$ defined  by $\pi(x):= (p_1(x),\ldots,p_{\mathrm{rk}\mathfrak{g}}(x))$,
\item for $X$ algebraic variety and for $x\in X$, $\mathcal{O}_{x,X}$ is the local ring of $X$ at $x$, $\mathfrak{M}_x$ is its maximal ideal and $\hat{\mathcal{O}}_{x,X}$ is the completion of $\mathcal{O}_{x,X}$ for the $\mathfrak{M}_x$-adic topology,
\item $\<.,.\>$ is the Killing form of $\mathfrak{g}$,
\item for $x\in\mathfrak{b}$, $x=x_0+x_+$ with $x_0\in\mathfrak{h}$ and $x_+\in\mathfrak{u}$,
\item $h$ is the element of $\mathfrak{h}$ such that $\beta(h)=1$, for all $\beta$ in $\Pi$ and $h(t)$ is the one parameter subgroup of $G$ generated by $\mathrm{ad}h$. Then for all $x$ in $\mathfrak{b}$, $$\lim\limits_{t\to0}h(t)(x)=x_0.$$
\end{itemize}

{\bf Acknowledgements}. I would like to thank Jean-Yves Charbonnel for his advice and help and for his rigorous attention to this work.

\section{On the variety $\sigma(\mathfrak{h}\times\mathfrak{h})$}

Denote by $(x,y)\longmapsto\overline{(x,y)}$ the canonical map from $\mathfrak{h}\times\mathfrak{h}$ to $(\mathfrak{h}\times\mathfrak{h})/W$.

\begin{prop} The set $X:=\sigma(\mathfrak{h}\times\mathfrak{h})$ is closed in $\Bbbk^{b_\mathfrak{g}+\rk\mathfrak{g}}$.
\end{prop}
\begin{proof} For $m:=\sup\{d_i,i\in\{i,\ldots,\mathrm{rk}\mathfrak{g}\}\}$, let $(t_i)_{i=1,\ldots,m+1}$ be in $\Bbbk$, pairwise different, let $\alpha$ be the morphism from $\mathfrak{h}\times\mathfrak{h}$ to $\mathfrak{h}^{m+1}$ defined by $$\alpha(x,y):=(x+t_1y,\ldots,x+t_{m+1}y)$$ and let $\gamma$ be the morphism from $\mathfrak{h}^{m+1}$ to $\Bbbk^{(m+1)\mathrm{rk}\mathfrak{g}}$ defined by $$\gamma(x_1,\ldots,x_{m+1}):=(\pi(x_1),\ldots,\pi(x_{m+1})).$$ Let $\beta$ be the morphism from $\Bbbk^{b_\mathfrak{g}+\mathrm{rk}\mathfrak{g}}$ to $\Bbbk^{(m+1)\mathrm{rk}\mathfrak{g}}$ defined by  $$\beta({(z_i^{(j)})}_{ 1\leq i\leq \mathrm{rk}\mathfrak{g}, 0\leq j\leq d_i}):= \left( \sum\limits_{j=0}^{d_1} t_1^j z_1^{(j)},\sum\limits_{j=0}^{d_2} t_1^j z_2^{(j)},\ldots , \sum\limits_{j=0}^{d_{\mathrm{rk}\mathfrak{g}}}t_{m+1}^jz_{\mathrm{rk}\mathfrak{g}}^{(j)}\right).$$
Since $(t_i)_{i=1,\ldots,m+1}$ are pairwise different then $\beta$ is an isomorphism from $\Bbbk^{b_\mathfrak{g}+\mathrm{rk}\mathfrak{g}}$ onto $\beta (\Bbbk^{b_\mathfrak{g}+\mathrm{rk}\mathfrak{g}})$. Since $\pi$ is a finite morphism, $\gamma$ is too. Hence $\gamma\circ\alpha(\mathfrak{h}\times\mathfrak{h})$ is closed in $\Bbbk^{(m+1)\mathrm{rk}\mathfrak{g}}$. Furthermore, $\gamma\circ\alpha(\mathfrak{h}\times\mathfrak{h})$ is contained in the image of $\beta$. Then $\sigma(\mathfrak{h}\times\mathfrak{h})$ is closed since $\beta\circ\sigma(\mathfrak{h}\times\mathfrak{h})=\gamma\circ\alpha(\mathfrak{h}\times\mathfrak{h})$
\end{proof}

\begin{prop}\label{P}: Let $(X_\n,\mu)$ be the normalization of $X=\sigma(\mathfrak{h}\times\mathfrak{h})$.
\begin{itemize}
\item[(i)] There exists a bijection $\bar{\sigma}$ from $\mathfrak{h}\times\mathfrak{h}\slash W$ to $X$ such that $\bar{\sigma}(\overline{(x,y)})=\sigma(x,y)$ for all $(x,y)$ in $\mathfrak{h}\times \mathfrak{h}$.
\item[(ii)] There exists a unique morphism $\sigma_\n$ from $(\mathfrak{h}\times\mathfrak{h})\slash W$ to $X_\n$ such that\linebreak $\mu\circ\sigma_\mathrm{n}=\bar{\sigma}$.
\item[(iii)] The morphism $\sigma_\n$ is an isomorphism.
\item[(iv)] The morphism $\mu$ is bijective.
\end{itemize}
\end{prop}
\begin{proof} (i) Since $\sigma(w(x,y))=\sigma(x,y)$, for all $w$ in $W$ and $(x,y)$ in $\mathfrak{h}\times\mathfrak{h}$, $\overline{\sigma}$ is well defined and it is surjective. Let $(x,y)$, $(x',y')$ be in $\mathfrak{h}\times\mathfrak{h}$ such that $\bar{\sigma}(\overline{(x,y)})=\bar{\sigma}(\overline{(x',y')})$. Then:\\ $$\bar{\sigma}(\overline{(x,y)})=\bar{\sigma}(\overline{(x',y')})\Rightarrow \sigma(x,y)=\sigma(x',y')\atop\Rightarrow p_i(x+ty)=p_i(x'+ty'),\mbox{ } \forall t\in \Bbbk\mbox{, } \forall i\in \{1,\ldots,rk\mathfrak{g}\}.$$ \\ Since $W$ is finite then there exist $t_1$, $t_2$ two different elements of $\Bbbk$ and $w$ in $W$ such that we have the following system:
$$
\left \{
\begin{array}{l}
   x'+t_1y'=w(x)+t_1w(y)\\
   x'+t_2y'=w(x)+t_2w(y)
\end{array}    
\right .$$
then $ x'=w(x) \mbox{ et } y'=w(y)$, hence $ \overline{(x',y')}=\overline{(x,y)},$ then $\bar{\sigma}$ is injective and therefore it is bijective.\\

(ii) The variety $(\mathfrak{h}\times\mathfrak{h})\slash W$ is a quotient of a smooth variety by a finite group, so it is normal. Since $\bar\sigma$ is bijective, it is a dominant morphism. Since $(\mathfrak{h}\times\mathfrak{h})\slash W$ is normal and since $\bar\sigma$ is a dominant morphism, there exists a unique morphism $\sigma_\n$ from $(\mathfrak{h}\times\mathfrak{h})\slash W$ into $X_\n$ such that $\mu\circ\sigma_\n=\bar\sigma$ (\cite{10} [ch. II, Ex. 3.8]). Then there is a commutative diagram:
$$ \xymatrix{
    \mathfrak{h}\times \mathfrak{h}\slash W  \ar[r]^{\sigma_\n} \ar[rd]_{\bar\sigma} & X_\n\ar[d]^{\mu}\\
     & X .
       }
  $$

(iii) Let $\Psi$ be the comorphism of $\sigma_\mathrm{n}$ from $\Bbbk[X_\mathrm{n}]$ into $S(\mathfrak{h}\times\mathfrak{h})^W$.\\
Show that $\Psi$ admits an inverse. Let $\Phi$ be the map from $S(\mathfrak{h}\times\mathfrak{h})^W$ to $\Bbbk[X_\mathrm{n}]$ defined by $\Phi(\mathrm{Q})=P$, with $\mathrm{P}(z) =\mathrm{Q}(x,y)$,
whenever $(x,y)$ in $\mathfrak{h}\times \mathfrak{h}$ and such that $\sigma_\mathrm{n}(x,y)=z$.
The map $\Phi$ is well defined. In fact, let $(x,y)$ and $(x',y')$ be in $\mathfrak{h}\times \mathfrak{h}$ such that $\sigma_\mathrm{n}(x,y)=\sigma_\n(x',y')$. We have to prove:$$\mathrm{Q}(x,y)=\mathrm{Q}(x',y')$$
We have:
$$\sigma_\n(x,y)=\sigma_\n(x',y')\Rightarrow \mu\circ\sigma_\n(x,y)=\mu\circ\sigma_\n(x',y')\Rightarrow\sigma(x,y)=\sigma(x',y').$$
It has been proved that there exists $w$ in $W$ such that $x'=w(x)$ and $y'=w(y)$ then $$ \mathrm{Q}(x',y')= \mathrm{Q}\left( w(x),w(y)\right)= \mathrm{Q}(x,y).$$ It is clear that the map $\Phi$ is an algebra homomorphism.
We now prove that $\mathrm{P}$ belongs to $\Bbbk[X_\mathrm{n}]$. Let $\Gamma$ be the graph of $\mathrm{Q}$ and let $\tilde{\Gamma}$ be the image of $\Gamma$ under the $\sigma_\mathrm{n} \times 1_{\mathbb{K}}$. Then $\tilde{\Gamma}$ is the graph of $\mathrm{P}$ and it is closed since $\sigma_\mathrm{n}\times 1_\Bbbk$ is a finite morphism and $\Gamma$ is closed. Let $\chi$ be the projection of $\tilde{\Gamma}$ onto $X_\mathrm{n}$.
Then $\chi$ is a bijection. Then by Zariski's main theorem (\cite{21}) $\chi$ is an isomorphism since $X_\mathrm{n}$ is normal. 
Hence $\mathrm{P}$ is a regular fonction on $X_\mathrm{n}$.\\
Let us show that $\Phi=\Psi^{-1}$. Let $\mathrm{P}$ be in $\Bbbk[X_\mathrm{n}]$ and let  $\mathrm{Q}$ be in $S(\mathfrak{h}\times\mathfrak{h})^W$. We have: $$\Phi\circ\Psi(\mathrm{P})=\Phi(P\circ\sigma_\mathrm{n})=\mathrm{P}\mbox{ and } \Psi\circ\Phi(\mathrm{Q})=\Phi(\mathrm{Q})\circ\sigma_\mathrm{n}.$$
Now calculate $(\Phi(\mathrm{Q})\circ\sigma_\mathrm{n})(x,y)$, for $(x,y)$ in $\mathfrak{h}\times\mathfrak{h}$,$$(\Phi(\mathrm{Q})\circ\sigma_\mathrm{n})(x,y)=\Phi(\mathrm{Q})(\sigma_\n(x,y))=\mathrm{Q}(x,y),$$ then $\Psi\circ\Phi(\mathrm{Q})=\mathrm{Q}$. 
Hence $\Psi$ is an isomorphism from $\Bbbk[X_\mathrm{n}]$ to $S(\mathfrak{h}\times\mathfrak{h})^W$, whence the assertion.\\

(iv) Since $\mu\circ\sigma_\n=\bar\sigma$, since $\sigma_\n$ is an isomorphism and since $\bar\sigma$ is bijective, $\mu$ is bijective.
\end{proof}

\begin{rem} By \cite{20}, $\bar\sigma$ is an isomorphism if and only if the algebra $S(\mathfrak{h}\times\mathfrak{h})^W$ is generated by the $2$-polarizations of elements of $S(\mathfrak{h})^W$, i.e. for $\mathfrak{g}$ of type $A_n$, $B_n$ and $C_n$. For $\mathfrak{g}$ of type $D_n$, some generators of $S(\mathfrak{h}\times\mathfrak{h})^W$ are given, and not all are polarizations. In particular for $n$ even, all the polarizations of elements of $S(\mathfrak{h})^W$ have even total degree and $S(\mathfrak{h}\times \mathfrak{h})^W$ contains elements of odd total degree. So $\Bbbk[X_\n]$ strictly contains $\Bbbk[X]$.
\end{rem}

\begin{prop} : Let $\Gamma:=\{(x,y)\in\mathfrak{h}\times\mathfrak{h}|P_{x,y}\cap\mathfrak{h'}\neq \varnothing \}$, we have:
\begin{itemize}
\item[(i)] $\mathrm{codim}_X\sigma(\Gamma)^c\geq 2$,
\item[(iii)] for all $z$ in $\sigma(\Gamma)$, $z$ is a regular point  of $X$.
\end{itemize}
\end{prop}
\begin{proof} (i) Let $(x,y)$ be in $\Gamma^c$, We have:$$(x,y)\in \Gamma^c \Rightarrow P_{x,y}\cap \mathfrak{h'}=\varnothing\Rightarrow x, y \in \mathfrak{h}\backslash \mathfrak{h'}.$$ Then $\Gamma^c$ is included in $\mathfrak{h}\backslash \mathfrak{h'} \times \mathfrak{h}\backslash \mathfrak{h'}.$
Hence $\mathrm{codim}_{\mathfrak{h}\times \mathfrak{h}}\Gamma^c\geqslant2$ and since $\sigma$ is a finite morphism, then $$\mathrm{codim}_X \sigma(\Gamma)^c\geq 2.$$

(ii) Let $(x,y)$ be an element in $\Gamma$ and let $(t_i)_{i=1,\ldots,m+1}$ be in $\k$, pairwise different, chosen such that $x+t_iy$ is regular for all $i$,  with $m= \sup\{d_i,i\in\{1,\ldots,\mathrm{rk}\mathfrak{g}\}\}$. Let $\beta$ be the morphism from $\Bbbk^{b_\mathfrak{g}+\mathrm{rk}\mathfrak{g}}$ to $\Bbbk^{(m+1)\mathrm{rk}\mathfrak{g}}$ defined by  $$\beta({(z_i^{(j)})}_{ 1\leq i\leq \mathrm{rk}\mathfrak{g}, 0\leq j\leq d_i})= \left( \sum\limits_{j=0}^{d_1} t_1^j z_1^{(j)},\sum\limits_{j=0}^{d_2} t_1^j z_2^{(j)},\ldots , \sum\limits_{j=0}^{d_{\mathrm{rk}\mathfrak{g}}}t_{m+1}^jz_{\mathrm{rk}\mathfrak{g}}^{(j)}\right).$$
Since $(t_i)_{i=1,\ldots,m+1}$ are pairwise different then $\beta$ is an isomorphism from $\Bbbk^{b_\mathfrak{g}+\mathrm{rk}\mathfrak{g}}$ onto $\beta (\Bbbk^{b_\mathfrak{g}+\mathrm{rk}\mathfrak{g}})$.
Moreover, $\beta(\sigma(\mathfrak{h}\times\mathfrak{h}))$ is equal to $\tau(\mathfrak{h}\times\mathfrak{h})$, where $\tau$ is the morphism from $\mathfrak{h}\times\mathfrak{h}$ to  $\Bbbk^{(m+1)\mathrm{rk}\mathfrak{g}}$ defined by $$\tau(x,y):=(\pi(x+t_1y),\ldots,\pi(x+t_{m+1}y)).$$ 
Let $\tau_{1,2}$ be the morphism from $\mathfrak{h}\times\mathfrak{h}$ to $\Bbbk^{2\mathrm{rk}\mathfrak{g}}$ defined by $$\tau_{1,2}(x,y):= \left( \pi(x+t_1y),\pi(x+t_2y)\right),$$ let $\varphi$ be the projection from $\tau(\mathfrak{h}\times\mathfrak{h})$ to $\tau_{1,2}(\mathfrak{h}\times\mathfrak{h})$ defined by $$\varphi \left( \pi(x+t_1y),\ldots, \pi(x+t_{m+1}y) \right):=(\pi(x+t_1y), \pi(x+t_2y))$$ and let $(\tau_{1,2})_*$ and $\varphi_*$ be the comorphisms of $\tau_{1,2}$ and $\varphi$ respectively. Then there is the following diagram:
$$ \xymatrix{
   \hat{\mathcal{O}}_{\tau(x,y),\tau(\mathfrak{h}\times\mathfrak{h})}  \ar[r]^{\tau_*}  & \hat{\mathcal{O}}_{(x,y),\mathfrak{h}\times\mathfrak{h}} \ar[d]^{(\tau_{1,2})_*^{-1}}\\
     & \hat{\mathcal{O}}_{\tau_{1,2}(x,y),\tau_{1,2}(\mathfrak{h}\times\mathfrak{h})} \ar[lu]_{\varphi_*}.
  }
  $$
Indeed, since $x+t_iy$ is a regular point, $\pi$ is an etale morphism  at $x+t_iy$ for $i=1,2$, hence $(\tau_{1,2})_*$ is an isomorphism. Let $\alpha$ be the morphism from $\mathfrak{h}\times\mathfrak{h}$ to $\mathfrak{h}^{m+1}$ defined by $$\alpha(x,y):=(x+t_1y,\ldots,x+t_{m+1}y)$$ and let $\gamma$ be the morphism from $\mathfrak{h}^{m+1}$ to $\k^{(m+1)\rk\mathfrak{g}}$ defined by $$\gamma(x_1,\ldots,x_{m+1}):=(\pi(x_1),\ldots,\pi(x_{m+1})).$$ Then there is the following commutative diagram:
$$ \xymatrix{
    \mathfrak{h}\times\mathfrak{h} \ar[r]^\alpha \ar[d]_{\tau_{1,2}} & \alpha(\mathfrak{h}\times\mathfrak{h})\ar[d]^\gamma\\
   \tau_{1,2}(\mathfrak{h}\times\mathfrak{h})  & \tau(\mathfrak{h}\times\mathfrak{h})\ar[l]_\varphi
  }.
  $$
Let $\alpha_*$ and $\gamma_*$ be the comorphisms of $\alpha$ and $\gamma$ respectively. Since $\alpha$ is an isomorphism, $(\alpha_*)^{-1}\circ(\tau_{1,2})_*$ is an isomorphism. Hence the coordinates on $\alpha(\mathfrak{h}\times\mathfrak{h})$ are formal series of coordinates on $\hat{\mathcal{O}}_{\tau_{1,2}(x,y),\tau_{1,2}(\mathfrak{h}\times\mathfrak{h})}$. Then the $p_i(x+t_jy)$'s are in the image of $\varphi_*$. Since the $p_i(x+t_jy)$'s generate $\hat{\mathcal{O}}_{\tau(x,y),\tau(\mathfrak{h}\times\mathfrak{h})}$, $\varphi_*$ is surjectif. Since $\hat{\mathcal{O}}_{\tau_{1,2}(x,y),\tau_{1,2}(\mathfrak{h}\times\mathfrak{h})}$ and $\hat{\mathcal{O}}_{\tau(x,y),\tau(\mathfrak{h}\times\mathfrak{h})}$ have same Krull's dimension $2\rk\mathfrak{g}$, $\ker\varphi_*=\{0\}$. Then $\varphi_*$ is an isomorphism. Hence $\hat{\mathcal{O}}_{\tau(x,y),\tau(\mathfrak{h}\times\mathfrak{h})}$ is a regular local ring, so that $\tau(x,y)$ is a smooth point of $\tau(\mathfrak{h}\times\mathfrak{h})$. Hence $\sigma(x,y)$ is a smooth point of $\sigma(\mathfrak{h}\times\mathfrak{h})$ since $\beta$ is an isomorphism from $\sigma(\mathfrak{h}\times\mathfrak{h})$ to $\tau(\mathfrak{h}\times\mathfrak{h})$. So for $(x,y)$ in $\Gamma$, $\sigma(x,y)$ is regular in $\sigma(\mathfrak{h}\times\mathfrak{h})$.
\end{proof}

We recall that the commuting variety $\mathcal{C}_\mathfrak{g}$ of $\mathfrak{g}$ is the set of $(x,y)$ in $\mathfrak{g}\times\mathfrak{g}$ such that $[x,y]=0$.

\begin{prop}: The variety $X$ is the image of the commuting variety $\mathcal{C}_\mathfrak{g}$ of $\mathfrak{g}$ by $\sigma$.
\end{prop}
\begin{proof} Let $(x,y)$ be in $\mathcal{C}_\mathfrak{g}$ and let $x=x_s+x_n$ and $y=y_s+y_n$ be the Jordan decompositions of $x$ and $y$. Since $[x,y]=0$, $[x_s,y_s]=0$. Then $x_s$ and $y_s$ belong to a same Cartan subalgebra of $\mathfrak{g}$. Since the Cartan subalgebras are conjuguate under $G$, there exists $g$ in $G$ such that $g(x_s)$ and $g(y_s)$ are in $\mathfrak{h}$. Since $$p_i(x+ty)=p_i(x_s+ty_s), \forall i=1,\ldots,\mathrm{rk}\mathfrak{g}$$ and $\sigma$ is $G$-invariant, $$\sigma(x,y)=\sigma(g(x_s),g(y_s)).$$ Hence $\sigma(\mathcal{C}_\mathfrak{g})$ is contained in $X$. Since $\mathcal{C}_\mathfrak{g}$ contains $\mathfrak{h}\times\mathfrak{h}$, $X$ is equal to $\sigma(\mathcal{C}_\mathfrak{g})$.
\end{proof}

\section{On the subvariety $\mathcal{B}_\mathfrak{g}$ of $\mathfrak{g}\times\mathfrak{g}$}

We consider the action of $\mathbf{B}$ on $G\times\mathfrak{b}\times\mathfrak{b}$ given by $b.(g,x,y)=(gb^{-1},b(x),b(y))$.
Let $\gamma'$ be the morphism from $G\times\mathfrak{b}\times\mathfrak{b}$ to $\mathfrak{g}\times\mathfrak{g}$ defined by  $$\gamma'(g,x,y)=(g(x),g(y))$$ and let $\gamma$ be the morphism from $G\times_\mathbf{B}\mathfrak{b}\times\mathfrak{b}$ to $\mathfrak{g}\times\mathfrak{g}$ defined through the quotient by $\gamma'$. Let $\mathcal{B}_\mathfrak{g}$ be the set of $(x,y)$ in $\mathfrak{g}\times \mathfrak{g}$ such that $x$ and $y$ are in the same Borel subalgebra. Then $\mathcal{B}_\mathfrak{g}$ is the image of $G\times_\mathbf{B}\mathfrak{b}\times\mathfrak{b}$ by $\gamma$.

\begin{prop}: The morphism $\gamma$ is a desingularization of $\mathcal{B}_\mathfrak{g}$. The subvariety ${\mathcal B}_{\mathfrak{g}}$ is closed and irreducible of dimension $3b_{\mathfrak{g}}-\mathrm{rk}\mathfrak{g}$. Moreover it is not normal.
\end{prop}
\begin{proof} Since $\mathcal{B}_\mathfrak{g}$ is the image of $G\times_\mathbf{B}\mathfrak{b}\times\mathfrak{b}$ by $\gamma$ and since $G/\mathbf{B}$ is a projective variety, ${\mathcal B}_{\mathfrak{g}}$ is closed and ${\mathcal B}_{\mathfrak{g}}$ is irreducible as the image of an irreducible variety.\\ Since $G\times_\mathbf{B}\mathfrak{b}\times\mathfrak{b}$ is a vector bundle on the smooth variety $G/\mathbf{B}$ then it is a smooth variety. Let $\mathcal{B}$ be the subvariety of the Borel subalgebras of $\mathfrak{g}$ and let $\tilde{\mathcal{B}}$ be the subvariety of elements $(u,x,y)$ of $\mathcal{B}\times\mathfrak{g}\times\mathfrak{g}$ such that $u$ contains $x$ and $y$. Hence there exists an isomorphism $\upsilon$ from $G\times_\mathbf{B}\mathfrak{b}\times\mathfrak{b}$ onto $\tilde{\mathcal{B}}$ such that $\gamma$ is the compound of $\upsilon$ and the canonical projection of $\tilde{\mathcal{B}}$ on $\mathfrak{g}\times\mathfrak{g}$. Whence $\gamma$ is a proper morphism since $G/\mathbf{B}$ is a projective variety.\\
It remains to show that $\gamma$ is a birational morphism. Let be $$\Omega_\mathfrak{g}:=\{(x,y)\in\mathfrak{g}\times\mathfrak{g}|P_{x,y}\backslash\{0\}\subset\mathfrak{g}'\mbox{, }\dim P_{x,y}=2\}.$$ The subset $\Omega_\mathfrak{g}$ is an open subset of $\mathfrak{g}\times\mathfrak{g}$ and $\Omega_\mathfrak{g}\cap\mathcal{B}_\mathfrak{g}$ is a nonempty open subset of $\mathcal{B}_\mathfrak{g}$. Let $(x,y)$ be an element in $\Omega_\mathfrak{g}\cap\mathcal{B}_\mathfrak{g}$. According to \cite{18} (Corollary 2) and \cite{28} (Theorem 9), the subspace $V_{x,y}$ generated by $\{\varepsilon_i^{(m)}(x,y), i=1,\ldots,\mathrm{rk}\mathfrak{g},m=0,\ldots,d_i\}$ is the unique Borel subalgebra which contains $x$ and $y$; hence $\upsilon^{-1}(V_{x,y},x,y)$ is the unique point in $G\times_\mathbf{B}\mathfrak{b}\times\mathfrak{b}$ above $(x,y)$. Whence $\gamma$ is birational.\\
 Since $\gamma$ is birational,$$\dim \mathcal{B}_\mathfrak{g}=\dim(G\times_\mathbf{B}\mathfrak{b}\times\mathfrak{b})=3b_\mathfrak{g}-\mathrm{rk}\mathfrak{g}.$$
Since $\gamma$ is a birational morphism, for $\mathfrak{g}$ non trivial, if ${\mathcal B}_{\mathfrak{g}}$ would be normal, by Zariski's main theorem a fiber of $\gamma$ would have cardinality 1 if it was finite, but $|\gamma^{-1}(x,0)|=|W|$ for $x$ in $\mathfrak{h}'$. So, $\mathcal{B}_\mathfrak{g}$ isn't normal.
\end{proof}

Let $({\mathcal B}_{\mathfrak{g},\mathrm{n}},\eta)$ be the normalization of ${\mathcal B}_{\mathfrak{g}}$.
 
\begin{Lemme} There exists a unique closed immersion $\iota_\n$ from $\mathfrak{h}\times\mathfrak{h}$ to ${\mathcal B}_{\mathfrak{g},\n}$ such that $\eta\circ\iota_\n=\iota$, whenever $\iota$ is the injection of $\mathfrak{h}\times\mathfrak{h}$ in $\mathcal{B}_\mathfrak{g}$.
\end{Lemme}
\begin{proof} Since $\gamma$ is a dominant morphism from $G\times_\mathbf{B}\mathfrak{b}\times\mathfrak{b}$ to $\mathcal{B}_\mathfrak{g}$ and since $G\times_\mathbf{B}\mathfrak{b}\times\mathfrak{b}$ is normal, there exists a unique morphism $\gamma_\n$ from $G\times_\mathbf{B}\mathfrak{b}\times\mathfrak{b}$ to ${\mathcal B}_{\mathfrak{g},\n}$ such that $\eta\circ\gamma_\n=\gamma$ (\cite{10} [ch. II, Ex. 3.8]). Then there is a commutative diagram:
$$ \xymatrix{
    G\times_\mathbf{B}\mathfrak{b}\times \mathfrak{b} \ar[r]^{\gamma_\n} \ar[rd]_{\gamma} & {\mathcal B}_{\mathfrak{g},\n} \ar[d]^{\eta}\\
     & \mathcal{B}_\mathfrak{g}.
       }
  $$
Let $i$ be the canonical injection of $\mathfrak{h}\times\mathfrak{h}$ into $G\times_\mathbf{B}\mathfrak{b}\times \mathfrak{b}$. Hence taking $\iota_\n=\gamma_\n\circ i$, $\eta\circ\iota_\n=\iota$ since $\iota=\gamma\circ i$. Then there is a commutative diagram:
$$ \xymatrix{
    \mathfrak{h}\times \mathfrak{h} \ar[r]^{\iota_\n} \ar[rd]_{\iota} & \iota_\n(\mathfrak{h}\times \mathfrak{h}) \ar[d]^{\eta}\\
     & \mathcal{B}_\mathfrak{g}.
       }
  $$
Since $\gamma_\n$ and $i$ are closed, $\iota_\n$ is closed.
\end{proof}

\begin{prop}: We have the following properties:
\begin{itemize}
 \item[(i)] $\sigma({\mathcal B}_{\mathfrak{g}})$ is equal to $X$,
 \item[(ii)] ${\mathcal B}_{\mathfrak{g}}$ is an irreducible component of $\sigma^{-1}(X)$.
 \end{itemize}
\end{prop}
\begin{proof} (i) Let $(x,y)$ be in ${\mathcal B}_{\mathfrak{g}}$. Since $\sigma$ is $G$-invariant, we can assume that $x$ and $y$ are in $\mathfrak{b}$. We have:$$p_i(x+ty)=p_i(x_0+ty_0)\mbox{, } \forall  i=1,\ldots, \mathrm{rk}\mathfrak{g}\atop\Rightarrow\sigma(x,y)=\sigma(x_0,y_0),$$ then $\sigma({\mathcal B}_{\mathfrak{g}})$ is contained in $X$, whence the equality since $\mathfrak{h}\times\mathfrak{h}$ is contained in $\mathcal{B}_\mathfrak{g}$.\\

(ii) For $z$ in $\Bbbk^{b_\mathfrak{g}+\mathrm{rk}\mathfrak{g}}$, by \cite{5}, the dimension of $\sigma^{-1}(z)$ is $3b_\mathfrak{g}-3\mathrm{rk}\mathfrak{g}$, then $$\dim\sigma^{-1}(X)=3b_\mathfrak{g}-\mathrm{rk}\mathfrak{g}= \dim {\mathcal B}_{\mathfrak{g}},$$ hence ${\mathcal B}_{\mathfrak{g}}$ is an irreducible component of $\sigma^{-1}(X)$.
\end{proof}

Let $\tau'$ be the morphism from $G\times\mathfrak{b}\times \mathfrak{b}$ to $\mathfrak{h}\times\mathfrak{h}$ defined by $$\tau'(g,x,y)=(x_0,y_0).$$ For $b$ in $\mathbf{B}$, 
\begin{eqnarray*}
\tau'(b.(g,x,y))&=&\tau'(gb^{-1},b(x),b(y))\\&=&(b(x)_0,b(y)_0)\\&=&(x_0,y_0),
\end{eqnarray*}
indeed, since $\mathbf{B}=\mathbf{U}\mathbf{H}$, there exists $u$ in $\mathbf{U}$ and $h$ in $\mathbf{H}$ such that $b=uh$.
Then $b(x)=b(x_0+x_+)=u(x_0)+b(x_+)$, $u(x_0)$ is in $x_0+\mathfrak{u}$ and $b(x_+)$ is in $\mathfrak{u}$, hence $b(x)$ is in $x_0+\mathfrak{u}$. Then $\tau'$ is constant on the $\mathbf{B}$-orbits. Hence there exists a morphism $\tau$ from $G\times_\mathbf{B}\mathfrak{b}\times\mathfrak{b}$ to $\mathfrak{h}\times\mathfrak{h}$ defined by $$\tau(\overline{(g,x,y)})=(x_0,y_0).$$

\begin{Lemme}\label{L2.2} Let $z$ and $z'$ be in $G\times_\mathbf{B}\mathfrak{b}\times\mathfrak{b}$ such that $\gamma(z')=\gamma(z)$. Then there exists $w$ in $W$ such that $\tau(z')=w.\tau(z)$.
\end{Lemme}
\begin{proof} Let $(g,x,y)$ and $(g',x',y')$ be in $G\times\mathfrak{b}\times\mathfrak{b}$ two representatives of $z$ and $z'$ respectively. We have $$\gamma(z)=\gamma(z')\Rightarrow(g'(x'),g'(y'))=(g(x),g(y))\atop\Rightarrow(x',y')=g'^{-1}g.(x,y).$$
Since $G=\underset{w\in W}{\bigcup}\mathbf{U}w\mathbf{B}$, there exists $u$ in $\mathbf{U}$, $w$ in $W$, $g_w$ in $\mathbf{N}_G(\mathfrak{h})$ a representative of $w$ and $b$ in $\mathbf{B}$ such that $g'^{-1}g=ug_wb$. Let $x=x_0+x_+$, $x'=x'_0+x'_+$, $y=y_0+y_+$ and $y'=y'_0+y'_+$, whenever $x_0$, $x'_0$, $y_0$ and $y'_0$ are in $\mathfrak{h}$ and $x_+$, $x'_+$, $y_+$ and $y'_+$ are in $\mathfrak{u}$. We have: $$(x',y')=(ug_wb(x),ug_wb(y))\atop\Rightarrow(u^{-1}(x'),u^{-1}(y'))=(g_wb(x),g_wb(y)).$$ Since $b(x)$ is in $x_0+\mathfrak{u}$, since $b(y)$ is in $y_0+\mathfrak{u}$, since $u^{-1}(x')$ is in $x'_0+\mathfrak{u}$ and since $u^{-1}(y')$ is in $y'_0+\mathfrak{u}$, there exists $v_x$, $v_y$, $v_{x'}$ and $v_{y'}$ in $\mathfrak{u}$, such that $$b(x)=x_0+v_x\mbox{, }b(y)=y_0+v_y,\atop u^{-1}(x')=x'_0+v_{x'}\mbox{ and } u^{-1}(y')=y'_0+v_{y'}.$$ Hence $$x'_0+v_{x'}=g_w(x_0)+g_w(v_x)\mbox{ and }\atop y'_0+v_{y'}=g_w(y_0)+g_w(v_y).$$ Since $v_x$ is in $\sum\limits_{\substack{\alpha\in\mathcal{R}}}\mathfrak{g}^\alpha$, $g_w(v_x)$ is in $\sum\limits_{\substack{\alpha\in \mathcal{R}}}\mathfrak{g}^{g_w.\alpha}=\sum\limits_{\substack{\alpha\in \mathcal{R}}}\mathfrak{g}^\alpha$. Then $g_w(v_x)-v_{x'}$ is in $\sum\limits_{\substack{\alpha\in \mathcal{R}}}\mathfrak{g}^\alpha$ and $g_w(x_0)=x'_0$, similarly $g_w(y_0)=y'_0$. Hence $\tau(z')=g_w.\tau(z)$.
\end{proof}

\begin{prop}\label{PSIBn} There exists uniquely defined morphism $\Phi$ from $S(\mathfrak{h}\times\mathfrak{h})^W$ to $\Bbbk[{\mathcal B}_{\mathfrak{g},\n}]^G$ such that $$\Phi(P)\circ\iota_\n=\mathrm{P}\mbox{, }\forall\mathrm{P}\in S(\mathfrak{h}\times\mathfrak{h})^W.$$ 
\end{prop}
\begin{proof} Let $\mathrm{P}$ be in $S(\mathfrak{h}\times\mathfrak{h})^W$, let $\Gamma$ be the graph of $\mathrm{P}\circ\tau$ and let $\Gamma'$ be the image of $\Gamma$ by $\gamma_\n\times id_\Bbbk$.
We want to prove that $\Gamma'$ is a graph of an element $\mathrm{Q}$ in $\Bbbk[{\mathcal B}_{\mathfrak{g},\n}]^G$. Let $(x,t)$ and $(x,t')$ be in $\Gamma'$. Let $z$ and $z'$ be in $G\times_\mathbf{B}\mathfrak{b}\times\mathfrak{b}$, such that $$(z,t)\mbox{, }(z',t')\in\Gamma\mbox{ and }\gamma_\n(z)=\gamma_\n(z')=x.$$ By lemma \ref{L2.2}, there exists $w$ in $W$ such that $\tau(z')=w.\tau(z)$ hence $$\mathrm{P}\circ\tau(z')=\mathrm{P}\circ\tau(z)\mbox{ and }t= t'.$$ Then $\Gamma'$ is a graph of a fonction $\mathrm{Q}$ on ${\mathcal B}_{\mathfrak{g},\n}$. Since $\Gamma$  and $\gamma_\n$ are closed, $\Gamma'$ is closed too. Let $\theta$ be the projection of $\Gamma'$ on ${\mathcal B}_{\mathfrak{g},\n}$. Since $\theta$ is bijective and since ${\mathcal B}_{\mathfrak{g},\n}$ is normal then by Zariski's main theorem $\theta$ is an isomorphism and hence $\mathrm{Q}$ is a regular fonction on ${\mathcal B}_{\mathfrak{g},\n}$. Since $\gamma_\n$ and $\Gamma$ are $G$-invariant, $\Gamma'$ is $G$-invariant and $\mathrm{Q}$ is $G$-invariant, then $\mathrm{Q}$ is in $\k[{\mathcal B}_{\mathfrak{g},\n}]^G$. 
Hence we have a morphism of algebra $\Phi$ from $S(\mathfrak{h}\times\mathfrak{h})^W$ to $\Bbbk[{\mathcal B}_{\mathfrak{g},\n}]^G$ such that $$\Phi(\mathrm{P})\circ\gamma_\n=\mathrm{P}\circ\tau.$$
Then, for $\mathrm{P}$ in $S(\mathfrak{h}\times\mathfrak{h})^W$,
\begin{eqnarray*}
\Phi(\mathrm{P})\circ\iota_\n&=&\Phi(\mathrm{P})\circ\gamma_\n\circ i\\&=&\mathrm{P}\circ\tau\circ i\\&=&\mathrm{P},
\end{eqnarray*}
since $\tau\circ i=id_{\mathfrak{h}\times\mathfrak{h}}$.
\end{proof}

\begin{prop}\label{P4}The morphism $\Phi$ is an isomorphism from $S(\mathfrak{h}\times\mathfrak{h})^W$ onto $\k[\mathcal{B}_\mathfrak{g}]^G$.
\end{prop}
\begin{proof} There are canonical morphisms $$\k[\mathcal{B}_\mathfrak{g}]^G\longrightarrow\k[\mathcal{C}_\mathfrak{g}]^G\mbox{,   }\k[\mathcal{C}_\mathfrak{g}]^G\longrightarrow S(\mathfrak{h}\times\mathfrak{h})^W,\atop S(\mathfrak{g}\times\mathfrak{g})^G\longrightarrow\k[\mathcal{B}_\mathfrak{g}]^G\mbox{ and }S(\mathfrak{g}\times\mathfrak{g})^G\longrightarrow\k[\mathcal{C}_\mathfrak{g}]^G\,$$ given by restrictions. Since $G$ is a reductive group, the two last arrows are surjective. So the first one is surjective and the second one is an isomorphism by (\cite{23}, Theorem 2.9). For all $(x,y)$ in $\mathfrak{b}\times\mathfrak{b}$, the intersection of the closure of $\mathbf{B}.(x,y)$ and $\mathfrak{h}\times\mathfrak{h}$ is nonempty. Indeed, according to Section 2, $$\lim\limits_{t\to0}h(t)(x,y)=(x_0,y_0).$$ So any $G$-orbit in $\mathcal{B}$ has a nonempty intersection with $\mathcal{C}_\mathfrak{g}$. As a consequence, the first arrow is an isomorphism since $G.(\mathfrak{h}\times\mathfrak{h})$ is dense in $\mathcal{C}_\mathfrak{g}$ by \cite{24}. Then the compound of the two first arrows is an isomorphism whose inverse is $\Phi$.
\end{proof}

\section{On the nullcone}

Let $\mathfrak{N}_\mathfrak{g}$ be the nilpotent cone and let $\mathfrak{N}'_\mathfrak{g}$ be the set of its regular points. We denote by $\mathcal{N}$ the set of $(x,y)$ in $\mathcal{B}_\mathfrak{g}$ such that $x$ and $y$ are nilpotents and $\mathcal{N}'$ the set of its smooth points. Let $\vartheta$ be the restriction of $\gamma$ to $G\times_\mathbf{B}\mathfrak{u}\times\mathfrak{u}$. Then $\mathcal{N}$ is the image of $G\times_\mathbf{B}\mathfrak{u}\times\mathfrak{u}$ by $\vartheta$.

\begin{prop} The morphism $\vartheta$ is a desingularization of $\mathcal{N}$. The variety $\mathcal{N}$ is closed and irreducible of dimension $3(b_\mathfrak{g}-\mathrm{rk}\mathfrak{g})$.
\end{prop}
\begin{proof} Since $G\times_\mathbf{B}\mathfrak{u}\times\mathfrak{u}$ is a vector bundle on the smooth variety $G/\mathbf{B}$, it is a smooth variety. There exists an embedding $\epsilon$ from $G\times_\mathbf{B}\mathfrak{u}\times\mathfrak{u}$ into $\tilde{\mathcal{B}}$ such that $\vartheta$ is the compound of $\epsilon$ and the canonical projection of $\tilde{\mathcal{B}}$ on $\mathfrak{g}\times\mathfrak{g}$. Whence $\vartheta$ is a proper morphism since $G/\mathbf{B}$ is a projective variety.\\
The set $\mathfrak{N}'_\mathfrak{g}\times\mathfrak{N}'_\mathfrak{g}\cap\mathcal{N}$ is an open subset of $\mathcal{N}$. Let $(x,y)$ be in $\mathfrak{N}'_\mathfrak{g}\times\mathfrak{N}'_\mathfrak{g}\cap\mathcal{N}$. Since $x$ is regular, there exists a unique Borel subalgebra containing $x$, so $\vartheta^{-1}(x,y)$ has a unique point. Hence $\vartheta$ is birational.\\ Since $\mathcal{N}$ is the image of $G\times_\mathbf{B}\mathfrak{u}\times\mathfrak{u}$ and since $G/\mathbf{B}$ is projective variety, $\mathcal{N}$ is closed. Since $G\times_\mathbf{B}\mathfrak{u}\times\mathfrak{u}$ is irreducible, $\mathcal{N}$ is irreducible. Since $\vartheta$ is birational,
\begin{eqnarray*}
\dim \mathcal{N}&=&\dim(G\times_\mathbf{B}\mathfrak{u}\times\mathfrak{u})\\&=&3(b_\mathfrak{g}-\mathrm{rk}\mathfrak{g}).
\end{eqnarray*}
\end{proof}
 
\begin{Lemme}\label{L4.1} Let $(x,y)$ be an element in $\mathcal{N}$ and let $(v,w)$ be in $T_{(x,y)}\mathcal{N}$. Then $v+tw$ is contained in $T_{x+ty}\mathfrak{N}$, for all $t$ in $\Bbbk$.
\end{Lemme}
\begin{proof} Let $p'_i$ be the differential of $p_i$ for $i=1,\ldots,\mathrm{rk}\mathfrak{g}$. Since $\mathfrak{N}=\pi^{-1}({0})$, it suffices to show that $p'_i(x+ty)(v+tw)=0$ for all $i=1,\ldots,\mathrm{rk}\mathfrak{g}$. We have 
\begin{eqnarray*}
p'_i(x+ty)(v+tw)&=&\frac d{da}p_i(x+ty+a(v+tw))|_{a=0}\\&=&\frac d{da}p_i(x+av+t(y+aw))|_{a=0}\\&=&\sum\limits_{m=0}^{d_i}t^m\frac d{da}p_i^{(m)}(x+av,y+aw)|_{a=0}.
\end{eqnarray*}
Since $(v,w)$ is in $T_{(x,y)}\mathcal{N}$ and since $\mathcal{N}\subset\sigma^{-1}(\{0\})$, $$\textstyle\frac d{da}p_i^{(m)}(x+av,y+aw)|_{a=0}=0,\forall m=0,\ldots,d_i-1\mbox{ and }i=1,\ldots,\mathrm{rk}\mathfrak{g}.$$ Then $$p'_i(x+ty)(v+tw)=0,\forall i=1,\ldots,\mathrm{rk}\mathfrak{g},$$whence the lemma.
\end{proof}

Let $t$ be in $\k$ and let the morphisms 
\begin{eqnarray*}
\theta'&:&G\times\mathfrak{u}\longrightarrow\mathfrak{N}\\& &(g,u)\mapsto g(u)\\ \rho&:&G\times\mathfrak{u}\times\mathfrak{u}\longrightarrow G\times\mathfrak{u}\\& &(g,u,u')\mapsto(g,u+tu')\\
\tau&:&\mathcal{N}\longrightarrow\mathfrak{N}\\& &\tau(x,y)\mapsto x+ty.
\end{eqnarray*}
Let $\vartheta'$ be the morphism from $G\times_\mathbf{B}\mathfrak{u}$ to $\mathfrak{N}$ defined through the quotient by $\theta'$ and let $\tau'$ be the morphism from $G\times_\mathbf{B}\mathfrak{u}\times\mathfrak{u}$ to $G\times_\mathbf{B}\mathfrak{u}$ defined through the quotient by $\rho$. Then we have the commutative diagram:
$$ \xymatrix{
      G\times_\mathbf{B}\mathfrak{u}\times\mathfrak{u}\ar[r]^{\vartheta} \ar[d]_{\tau'} & \mathcal{N} \ar[d]^{\tau}\\
      G\times_\mathbf{B}\mathfrak{u}\ar[r]_{\vartheta'}&\mathfrak{N}.
       }
  $$

\begin{Lemme}\label{L4.4} Let $(x,y)$ be in $\mathcal{N}$. If the intersection of $P_{x,y}$ and $\mathfrak{g}'$ is not empty, then $(x,y)$ is regular in $\mathcal{N}$.
\end{Lemme}
\begin{proof} It suffices to prove the lemma for $x$ regular. Let $(x,y)$ be in $\mathcal{N}$ such that $x$ is regular and let $(v,w)$ be in $T_{(x,y)}\mathcal{N}$. By Lemma \ref{L4.1}, $v+tw$ is in $T_{x+ty}\mathfrak{N}$ for all $t$ in $\Bbbk$. Let $t\neq0$ be in $\Bbbk$, such that $x+ty$ is regular. Since $x$ is regular, there exists $(\xi,\omega_1)$ in $T_{(1,x)}G\times_\mathbf{B}\mathfrak{u}$ such that $[\xi,x]+\omega_1=v$ and since $x+ty$ is regular, there exists $\omega_2$ in $\mathfrak{u}$ such that $[\xi,x+ty]+\omega_2=v+tw$. Then, for $\omega'_2=\textstyle\frac 1 t(\omega_2-\omega_1)$, $[\xi,y]+\omega'_2=w$. Hence, $$T_{(x,y)}\mathcal{N}\subset\{(v,w)\in\mathfrak{g}\times\mathfrak{g}|\mbox{ }\exists\mbox{ }\xi\in\mathfrak{g}\mbox{ and }\omega_1,\omega_2\in\mathfrak{u}\mbox{, verifying }\atop[\xi,x]+\omega_1=v\mbox{ and }[\xi,y]+\omega_2=w\}.$$ Let $\mu$ be the morphism from $\mathfrak{g}\times\mathfrak{u}\times\mathfrak{u}$ to $\mathfrak{g}\times\mathfrak{g}$ defined by$$\mu(\xi,\omega_1,\omega_2)=([\xi,x]+\omega_1,[\xi,y]+\omega_2).$$ Then $T_{(x,y)}\mathcal{N}$ is contained in the image of $\mu$. The set $\mu^{-1}(\{0\})$ is equal to the set of elements $(\xi,\omega_1,\omega_2)$, such that $[\xi,x]$ and $[\xi,y]$ are in $\mathfrak{u}$. Let $(\xi,\omega_1,\omega_2)$ be in $\mu^{-1}(\{0\})$. Since $x$ is regular, $[\xi,x]$ is in $\mathfrak{u}$ if and only if $\xi$ is in $\mathfrak{b}$. Hence $\dim\mu^{-1}(\{0\})=b_\mathfrak{g}$, whence
\begin{eqnarray*}
\dim T_{(x,y)}\mathcal{N}&\leqslant&\dim\mathfrak{g}\times\mathfrak{u}\times\mathfrak{u}-\dim \mu^{-1}(\{0\})\\&=&3(b_\mathfrak{g}-\mathrm{rk}\mathfrak{g})=\dim\mathcal{N}.
\end{eqnarray*}
As a result, $(x,y)$ is regular in $\mathcal{N}$.
\end{proof}

\begin{cor} The codimension of the set of singular points of $\mathcal{N}$ is at least four.
\end{cor}
\begin{proof} By Lemma \ref{L4.4} we have $$(\mathfrak{N}'_\mathfrak{g}\times\mathfrak{N}_\mathfrak{g}\cup\mathfrak{N}_\mathfrak{g}\times\mathfrak{N}'_\mathfrak{g}) \cap\mathcal{N}\subset\mathcal{N}',$$ then $$\mathcal{N}\backslash\mathcal{N}'\subset\mathcal{N}\backslash (\mathfrak{N}'_\mathfrak{g}\times\mathfrak{N}_\mathfrak{g}\cup\mathfrak{N}_\mathfrak{g}\times\mathfrak{N}'_\mathfrak{g}).$$ Since $\mathcal{N}\backslash (\mathfrak{N}'_\mathfrak{g}\times\mathfrak{N}_\mathfrak{g}\cup\mathfrak{N}_\mathfrak{g} \times\mathfrak{N}'_\mathfrak{g})$ is the image of $G\times_\mathbf{B}\mathfrak{u}\backslash\mathfrak{u}'\times\mathfrak{u}\backslash\mathfrak{u}'$, 
\begin{eqnarray*}
\dim\mathcal{N}\backslash\mathcal{N}'&\leqslant&\dim\mathcal{N}\backslash (\mathfrak{N}'_\mathfrak{g}\times\mathfrak{N}_\mathfrak{g}\cup\mathfrak{N}_\mathfrak{g}\times\mathfrak{N}'_\mathfrak{g})\\ &\leqslant&\dim G\times_\mathbf{B}\mathfrak{u}\backslash\mathfrak{u}'\times\mathfrak{u}\backslash\mathfrak{u}'\\
&=&3(b_\mathfrak{g}-\mathrm{rk}_\mathfrak{g})-4.
\end{eqnarray*}
\end{proof}

For $\alpha$ simple root, let $\mathfrak{p}_\alpha=\mathfrak{g}^{-\alpha}+\mathfrak{b}$ be the minimal parobolic subalgebra corresponding to $\alpha$ and let $L$ be the set of all Borel subalgebras contained in $\mathfrak{p}_\alpha$. Then $L$ is a projective line. Such a line will be called a projective line of type $\alpha$. For $(x,y)$ in $\mathcal{N}$ denote by $\mathcal{B}_{x,y}$ the set of Borel subalgebras containing $x$ and $y$.

\begin{Lemme} \label{L3.2} Let $(x,y)$ be in $\mathcal{N}$ and let $\mathfrak{b}$ and $\mathfrak{b}'$ be two elements in $\mathcal{B}_{x,y}$. There exist a sequence of projective lines $(L_i)_{1\leqslant i\leqslant q}$ and a sequence $(\mathfrak{b}_i)_{0\leqslant i\leqslant q}$ in $\mathcal{B}_{x,y}$ such that $\mathfrak{b}_0=\mathfrak{b}$, $\mathfrak{b}_q=\mathfrak{b}'$ and for all $i=1,\ldots,q$, $\mathfrak{b}_{i-1}$ and $\mathfrak{b}_i$ are in $L_i$.
\end{Lemme}
\begin{proof} The proof is inspired from \cite{15} (6.5). There exists $g$ in $G$ such that $\mathfrak{b}'=g(\mathfrak{b})$ and by Bruhat decomposition there exist $b$, $b'$ in $\mathbf{B}$, $w$ in $W$ and $n_w$ in $N_G(\mathfrak{h})$ representing $w$ such that $g=bn_wb'$. Since $b'$ is in $\mathbf{B}$, we can assume that $g=bn_w$. We will show the Lemma by induction on the length $l(w)$ of $w$.\\ Suppose that $l(w)=1$, then $w=s_\alpha$, for a simple root $\alpha$. Let $u=b^{-1}(x)$ and $v=b^{-1}(y)$. So $u$ and $v$ belong to $\mathfrak{u}\cap s_\alpha(\mathfrak{u})$, the nilpotent radical of the minimal parabolic subalgebra $\mathfrak{p}_\alpha$. Then $x$ and $y$ belong too, hence $x$ and $y$ are in all Borel subalgebras of $\mathfrak{p}_\alpha$, whence $$L_\alpha\subset\mathcal{B}_{x,y}\mbox{ and }\mathfrak{b}\mbox{, }\mathfrak{b}'\in L_\alpha.$$ Suppose the Lemma true for $l(w)\leqslant q-1$.\\ Let $w=s_1s_2\ldots s_q$ be a reduced expression, whenever each $s_i$ is the reflection associated to the simple root $\alpha_i$. Let $u=b^{-1}(x)$ and $v=b^{-1}(y)$, so $u$ and $v$ are in $\mathfrak{u}\cap w(\mathfrak{u})$. We have $\mathfrak{u}=\bigoplus\limits_{\gamma\in\mathcal{R}^+}\mathfrak{g}^\gamma$, then $w(\mathfrak{u})=\bigoplus\limits_{\gamma\in\mathcal{R}^+}\mathfrak{g}^{w(\gamma)}$. Hence $\mathfrak{u}\cap w(\mathfrak{u})=\bigoplus\limits_{\gamma\in\mathcal{R}^+ \atop w(\gamma)>0}\mathfrak{g}^\gamma$. Since $w(\alpha_q)<0$, $\mathfrak{u}\cap w(\mathfrak{u})\subset\mathfrak{u}\cap s_{\alpha_q}(\mathfrak{u})$, the nilpotente radical of $\mathfrak{p}_{\alpha_q}$, then $u$ and $v$ are in $\mathfrak{u}\cap s_{\alpha_q}(\mathfrak{u})$, hence $x$ and $y$ are too. In particular, $x$ and $y$ are in $s_q(\mathfrak{b})$. As a result, $\mathfrak{b}$ and $s_q(\mathfrak{b})$ are in $L_{\alpha_q}$ and $s_q(\mathfrak{b})$ and $bn_{w'}(s_q(\mathfrak{b}))$ are in $\mathcal{B}_{x,y}$, with $w'=s_1\ldots s_{q-1}$ and $n_{w'}\in N_G(\mathfrak{h})$ is a representative of $w'$. Since $l(w')=q-1$, by induction hypothesis, there exist a sequence of projective lines $(L_i)_{1\leqslant i\leqslant q-1}$ and a sequence $(\mathfrak{b}_i)_{0\leqslant i\leqslant q-1}$ in $\mathcal{B}_{x,y}$, such that $$\mathfrak{b}_0=s_q(\mathfrak{b})\mbox{, } \mathfrak{b}_{q-1}=bn_{w'}(s_q(\mathfrak{b}))\mbox{ and }\mathfrak{b}_{i-1},\mathfrak{b}_i\in L_i,\forall 1\leqslant i\leqslant q-1.$$ And we have $$\mathfrak{b}\mbox{, }s_q(\mathfrak{b})\in L_{\alpha_q}=L_0,$$ then the sequence of projective lines $(L_i)_{0\leqslant i\leqslant q-1}$ and the sequence $\mathfrak{b},\mathfrak{b}_0,\ldots,\mathfrak{b}_{q-1}$ verify $$\mathfrak{b}_{q-1}=\mathfrak{b}'\mbox{, }\mathfrak{b},\mathfrak{b}_0\in L_0\mbox{ and }\mathfrak{b}_{i-1},\mathfrak{b}_i\in L_i,\forall 1\leqslant i\leqslant q-1.$$ It was to be proven.
\end{proof}

\begin{prop} The set $\mathcal{B}_{x,y}$ is connected.
\end{prop}
\begin{proof} Let $\mathfrak{b}$ and $\mathfrak{b}'$ be two elements in $\mathcal{B}_{x,y}$ and let $X$ be the connected component of $\mathcal{B}_{x,y}$ containing $\mathfrak{b}$. By Lemma \ref{L3.2}, there exist a sequence of projective lines $(L_i)_{1\leqslant i\leqslant q}$ and a sequence $(\mathfrak{b}_i)_{0\leqslant i\leqslant q}$ in $\mathcal{B}_{x,y}$ such that $$\mathfrak{b}_0=\mathfrak{b}\mbox{, }\mathfrak{b}_q=\mathfrak{b}'\mbox{ and }\mathfrak{b}_{i-1},\mathfrak{b}_i\in L_i,\forall i=1,\ldots,q.$$ Since $\mathfrak{b}$ and $\mathfrak{b}_1$ are in $L_1$ and since $L_1$ is connected, $X$ contains $L_1$ and $\mathfrak{b}_1$. By induction on $q$, suppose that $X$ contains $\mathfrak{b}_{q-1}$. Since $\mathfrak{b}_{q-1}$ and $\mathfrak{b}_q$ are in $L_q$ and since $L_q$ is connected, $X$ contains $L_q$ and $\mathfrak{b}_q=\mathfrak{b}'$. Then $\mathcal{B}_{x,y}$ is connected.
\end{proof}

Let $(\mathcal{N}_\n,\tau)$ be the normalization of $\mathcal{N}$.

\begin{prop} The morphism $\tau$ is bijective.
\end{prop}
\begin{proof} Since $G\times_\mathbf{B}\mathfrak{u}\times\mathfrak{u}$ is normal and since $\vartheta$ is dominant, there exists a unique morphism $\vartheta_\n$ from $G\times_\mathbf{B}\mathfrak{u}\times\mathfrak{u}$ to $\mathcal{N}_\n$ such that $\vartheta=\tau\circ\vartheta_\n$. Then there is a commutative diagram:
$$ \xymatrix{
      G\times_\mathbf{B}\mathfrak{u}\times\mathfrak{u}\ar[r]^{\vartheta_\n} \ar[rd]_{\vartheta} & \mathcal{N}_\n\ar[d]^{\tau}\\
     & \mathcal{N} .
       }
  $$
Let $(x,y)$ be in $\mathcal{N}$. So $\vartheta^{-1}(x,y)=\vartheta_\n^{-1}(\tau^{-1}(x,y))$.
Since $\vartheta$ is a birational proper morphism, $\vartheta_\n$ is too. Then $\vartheta^{-1}(x,y)$ is the disjoint union of the fibres of $\vartheta_n$ at the elements of $\tau^{-1}(x,y)$. By Zariski's main theorem $\vartheta_\n^{-1}(z)$ is connected, for all $z$ in $\mathcal{N}_\n$. Since $\mathcal{B}_{x,y}$ is connected, $\vartheta^{-1}(x,y)$ is too. Whence $|\tau^{-1}(x,y)|=1$. As a result $\tau$ is bijective since $\tau$ is surjective.
\end{proof}

\section{Rational singularities}

In this section, we deeply use the ideas of \cite{11}. For $E$ a finite dimensional $\mathbf{B}$-module, we denote by $\mathcal{L}(E)$ the sheaf of local sections of the fiber bundle over $G/\mathbf{B}$ defined as the quotient of $G\times E$ under the right action of $\mathbf{B}$ given by $(g,v).b:=(gb,b^{-1}(v))$. Then $\mathcal{L}$ is a covariant exact functor from the category of finite dimensional $\mathbf{B}$-modules to the category of locally free $\mathcal{O}_{G/\mathbf{B}}$-modules of finite rank. Let $\Delta$ be the diagonal of $G/\mathbf{B}\times G/\mathbf{B}$ and let $\mathcal{J}_\Delta$ be its ideal of defenition in $\mathcal{O}_{G/\mathbf{B}\times G/\mathbf{B}}$. Since $\Delta$ is isomorphic to $G/\mathbf{B}$, $\mathcal{O}_{G/\mathbf{B}}$ is isomorphic to $\mathcal{O}_{G/\mathbf{B}\times G/\mathbf{B}}/\mathcal{J}_\Delta$. Let $(\mathfrak{v}_i)_{i=1,\ldots,j}$ be a sequence of subalgebras of $\mathfrak{b}$ containing $\mathfrak{u}$, whenever $j$ is a positive integer and let $E=\wedge^{m_1}\mathfrak{v}_1\otimes_\k\ldots\otimes_\k\wedge^{m_j}\mathfrak{v}_j$, whenever $(m_i)_{i=1,\ldots,j}$ is a sequence of nonnegative integers.

\begin{Lemme}\label{LC} Let $E_1$, $E_2$ be two $B$-modules. For an integer $i$, $$\mathrm{H}^i(G/\mathbf{B}\times G/\mathbf{B},\mathcal{L}(E_1)\boxtimes\mathcal{L}(E_2))=\bigoplus\limits_{l+k=i}\mathrm{H}^l(G/\mathbf{B},\mathcal{L}(E_1))\otimes_\k\mathrm{H}^k(G/\mathbf{B},\mathcal{L}(E_2)).$$
\end{Lemme}
\begin{proof} Let set $\mathcal{M}:=\mathcal{L}(E_1)\boxtimes\mathcal{L}(E_2)$. Let us consider flabby resolutions of $\mathcal{L}(E_1)$ and $\mathcal{L}(E_2)$, $$0\rightarrow\mathcal{L}(E_1)\rightarrow\mathcal{M}_0\rightarrow\mathcal{M}_1\rightarrow\ldots \atop0\rightarrow\mathcal{L}(E_2)\rightarrow\mathcal{N}_0\rightarrow\mathcal{N}_1\rightarrow\ldots.$$ For $i\geqslant0$, we denoted by $M_i$ and $N_i$ the spaces of global sections of $\mathcal{M}_i$ and $\mathcal{N}_i$ respectively. Then the cohomologies of $\mathcal{L}(E_1)$ and $\mathcal{L}(E_2)$ are the cohomologies of the complexes $$0\rightarrow\Gamma(G/\mathbf{B},\mathcal{L}(E_1))\rightarrow M_0\rightarrow M_1\rightarrow\ldots\atop 0\rightarrow\Gamma(G/\mathbf{B},\mathcal{L}(E_2))\rightarrow N_0\rightarrow N_1\rightarrow\ldots$$ respectively. From these two complexes we deduce a double complex $C^{\bullet,\bullet}$ whose underlying space is the direct sum of the spaces $M_i\otimes_\Bbbk N_j$ and the cohomology of the simple complex $C^\bullet$ deduced from $C^{\bullet,\bullet}$ is the cohomology of $\mathcal{M}$. The term $E_2^{p,q}$ of the spectral sequence of $C^{\bullet,\bullet}$ is $$E_2^{p,q}=\mathrm{H}^p(G/\mathbf{B},\mathcal{L}(E_1))\otimes_\Bbbk\mathrm{H}^q(G/\mathbf{B},\mathcal{L}(E_2)).$$ Then $$\mathrm{H}^i(C^\bullet)=\bigoplus\limits_{l+k=i}\mathrm{H}^{l}(G/\mathbf{B},\mathcal{L}(E_1))\otimes_\Bbbk \mathrm{H}^k(G/\mathbf{B},\mathcal{L}(E_2)),$$ whence the Lemma.
\end{proof}

\begin{prop}\label{PC} For $i>0$, $\mathrm{H}^{i+m_1+\ldots+m_j}(G/\mathbf{B},\mathcal{L}(E))=0.$
\end{prop}
\begin{proof} Let show that $$\mathrm{H}^{i+m_1+\ldots+m_j}(G/\mathbf{B},\mathcal{L}(E))=0\mbox{, for }i>0\mbox{  }(*)$$ by induction on $(n+1-i,j)$, whenever $n=\dim G/\mathbf{B}$.\\
For $n+1-i=0$, by \cite{10} (ch.III, Theorem 2.7) $\mathrm{H}^{i+m_1+\ldots+m_j}(G/\mathbf{B},\mathcal{L}(E))=0$ and for $j=1$, by \cite{11} (Theorem B) $(*)$ holds. Suppose that $(*)$ holds for pairs smaller than $(n+1-i,j)$, show it for $(n+1-i,j)$.\\ Let $E_1=\wedge^{m_1}\mathfrak{v}_1\otimes_\k\ldots\otimes_\k\wedge^{m_{j-1}}\mathfrak{v}_{j-1}$ and let $E'=\wedge^{m_j}\mathfrak{v}_j$. Then $\mathcal{L}(E_1)\boxtimes\mathcal{L}(E')$ is the sheaf of local sections of the fiber bundle over $G/\mathbf{B}\times G/\mathbf{B}$ defined as the quotient of $G\times G\times E_1\times E'$ under the right action of $\mathbf{B}\times\mathbf{B}$ given by $(g,g',v,v').(b,b')=(gb,g'b',b^{-1}(v),b'^{-1}(v'))$. There exists an exact sequence $$0\rightarrow\mathcal{J}_\Delta\rightarrow\mathcal{O}_{G/\mathbf{B}\times G/\mathbf{B}}\rightarrow \mathcal{O}_{G/\mathbf{B}}\rightarrow0,$$ whence the following exact sequence $$0\rightarrow\mathcal{J}_\Delta\otimes_{\mathcal{O}_{G/\mathbf{B}\times G/\mathbf{B}}}\mathcal{L}(E_1)\boxtimes\mathcal{L}(E')\rightarrow\mathcal{L}(E_1)\boxtimes\mathcal{L}(E')\rightarrow \mathcal{L}(E_1)\otimes\mathcal{L}(E')\rightarrow0.$$ Hence we will have the long exact sequence $$\cdots\rightarrow \mathrm{H}^i(G/\mathbf{B}\times G/\mathbf{B},\mathcal{J}_\Delta\otimes\mathcal{L}(E_1)\boxtimes\mathcal{L}(E')) \rightarrow\mathrm{H}^i(G/\mathbf{B}\times G/\mathbf{B},\mathcal{L}(E_1)\boxtimes\mathcal{L}(E'))\rightarrow\atop \mathrm{H}^i(G/\mathbf{B},\mathcal{L}(E_1)\otimes\mathcal{L}(E'))\rightarrow\mathrm{H}^{i+1}(G/\mathbf{B}\times G/\mathbf{B},\mathcal{J}_\Delta\otimes\mathcal{L}(E_1)\boxtimes\mathcal{L}(E'))\rightarrow\cdots.$$ By Lemma \ref{LC}, $$\mathrm{H}^{i+m_1+\ldots+m_j}(G/\mathbf{B}\times G/\mathbf{B},\mathcal{L}(E_1)\boxtimes\mathcal{L}(E')) =\bigoplus\limits_{l+k=i+\atop m_1+\ldots+m_j}\mathrm{H}^l(G/\mathbf{B},\mathcal{L}(E_1))\otimes_\k\mathrm{H}^k (G/\mathbf{B},\mathcal{L}(E')),$$ then, let $l=l'+m_1\ldots+m_{j-1}$,
\begin{itemize}
\item if $k>m_j$, by \cite{11} (Theorem B) $$\mathrm{H}^k(G/\mathbf{B},\mathcal{L}(E'))=0,$$
\item and if $k\leqslant m_j$, then 
\begin{eqnarray*}
l\geqslant i+m_1+\ldots+m_{j-1}&\Rightarrow&l'\geqslant i\\
&\Rightarrow&n+1-l'\leqslant n+1-i\\
&\Rightarrow&(n+1-l',j-1)<(n+1-i,j).
\end{eqnarray*}
Hence, by induction hypothesis $$\mathrm{H}^l(G/\mathbf{B},\mathcal{L}(E_1))=0.$$ 
\end{itemize}
Whence $$\mathrm{H}^{i+m_1+\ldots+m_j}(G/\mathbf{B}\times G/\mathbf{B},\mathcal{L}(E_1)\boxtimes \mathcal{L}(E'))=0\mbox{, for }i>0,$$ then it remains to show that $$\mathrm{H}^{i+1+m_1+\ldots+m_j}(G/\mathbf{B}\times G/\mathbf{B},\mathcal{J}_\Delta\otimes\mathcal{L}(E_1)\boxtimes\mathcal{L}(E'))=0\mbox{, for }i>0$$ Let identify $G/\mathbf{B}$ with $\Delta$. The sequence $\mathcal{J}_\Delta^k$, $k\in\mathbb{N}^*$, is a descending filtration of $\mathcal{J}_\Delta$. Let denote by $\mathcal{G}_\Delta$ the associated graded sheaf to this module. Then $\mathcal{G}_\Delta$ is a sheaf of graded algebra over $G/\mathbf{B}$. Moreover, the sequence $\mathcal{J}_\Delta^k\otimes_{\mathcal{O}_{G/\mathbf{B}\times G/\mathbf{B}}}\mathcal{L}(E_1)\boxtimes\mathcal{L}(E')$, $k\in\mathbb{N}^*$, is a descending filtration of $\mathcal{J}_\Delta\otimes_{\mathcal{O}_{G/\mathbf{B}\times G/\mathbf{B}}}\mathcal{L}(E_1)\boxtimes\mathcal{L}(E')$ and the associated graded sheaf is isomorphic to $\mathcal{G}_\Delta\otimes_{\mathcal{O}_{G/\mathbf{B}\times G/\mathbf{B}}}\mathcal{L}(E_1)\boxtimes\mathcal{L}(E')$ since $\mathcal{L}(E_1)\boxtimes\mathcal{L}(E')$ is locally free.

According to (\cite{10}, Ch. II, Theorem 8,17), since $\Delta$ is a smooth subvariety of $G/\mathbf{B}\times G/\mathbf{B}$, the following short sequence $$0\rightarrow\mathcal{J}_\Delta/\mathcal{J}_\Delta^2\rightarrow\Omega_{G/\mathbf{B}\times G/\mathbf{B}}^1\otimes_{\mathcal{O}_{G/\mathbf{B}\times G/\mathbf{B}}}\mathcal{O}_{G/\mathbf{B}}\rightarrow \Omega_{G/\mathbf{B}}^1\rightarrow0$$ is exact. Since $\Omega^1_{G/\mathbf{B}}$ is canonically isomorphic to $\mathcal{L}(\mathfrak{u}^*)$, $\mathcal{J}_\Delta/\mathcal{J}_\Delta^2$ is canonically isomorphic to $\mathcal{L}(\mathfrak{u}^*)$. Then, according to (\cite{10}, Ch. II, Theorem 8,21A), $\mathcal{G}_\Delta$ is isomorphic to $\mathcal{L}(\mathrm{S}(\mathfrak{u}^*))$ since $\Delta$ is locally a complete intersection in $G/\mathbf{B}\times G/\mathbf{B}$. As a result, the associated graded module to the filtration on $\mathcal{G}_\Delta\otimes_{\mathcal{O}_{G/\mathbf{B}\times G/\mathbf{B}}}\mathcal{L}(E_1)\boxtimes\mathcal{L}(E')$ is ismorphic to the $\mathcal{O}_{G/\mathbf{B}}$-module $$\mathcal{L}(S(\mathfrak{u}^*)) \otimes_{\mathcal{O}_{G/\mathbf{B}}}\mathcal{L}(E_1)\otimes_{\mathcal{O}_{G/\mathbf{B}}}\mathcal{L}(E').$$ Hence it suffice to prove that $$\mathrm{H}^{i+1+m_1+\ldots+m_j}(G/\mathbf{B},\mathcal{L}(\mathrm{S}(\mathfrak{u}^*)) \otimes_{\mathcal{O}_{G/\mathbf{B}}}\mathcal{L}(E_1)\otimes_{\mathcal{O}_{G/\mathbf{B}}}\mathcal{L}(E'))=0\mbox{, for }i>0.$$ According to the identification of $\mathfrak{g}$ with its dual by the killing form, one has a short exact sequence $$0\rightarrow\mathfrak{b}\rightarrow\mathfrak{g}^* \rightarrow\mathfrak{u}^*\rightarrow0.$$ From this exact sequence, on deduces the exact Koszul comlpex $$\cdots\xrightarrow{\mathrm{d}}K_2\xrightarrow{\mathrm{d}}K_1 \xrightarrow{\mathrm{d}}K_0\rightarrow\mathrm{S}(\mathfrak{u}^*)\rightarrow0,$$ with $$K_p:=\mathrm{S}(\mathfrak{g}^*)\otimes_\k\wedge^p(\mathfrak{b}),$$ $$\mathrm{d}(a\otimes(a_0\wedge\ldots\wedge a_p)):=\sum\limits^{p}_{i=0}(-1)^ia_ia\otimes(a_0\wedge\ldots\wedge\hat{a}_i\wedge\ldots\wedge a_p).$$ This complex $K_\bullet$ is canonically graded by $$K_\bullet:=\displaystyle\sum_{q}K_\bullet^q\mbox{, where }K_p^q:=\mathrm{S}^{q-p}(\mathfrak{g^*})\otimes_\Bbbk \wedge^p(\mathfrak{b}).$$ So that the sequences $$\ldots\rightarrow K_2^q\rightarrow K_1^q\rightarrow K_0^q\rightarrow\mathrm{S}^q(\mathfrak{u}^*)\rightarrow 0\atop\cdots\rightarrow \tilde{K}_2^q\rightarrow\tilde{K}_1^q\rightarrow\tilde{K}_0^q\rightarrow\mathrm{S}(\mathfrak{u}^*)\otimes_\k E_1\otimes_\k E'\rightarrow0$$ are exact with $$\tilde{K}_p^q:=K_p^q\otimes_\k E_1\otimes_\k E',$$ for all $p$. Then, since $\mathcal{L}$ is an exact functor, one deduces the exact sequence of $\mathcal{O}_{G/\mathbf{B}}$-modules $$\cdots\rightarrow \mathcal{L}(\tilde{K}_2^q)\rightarrow\mathcal{L}(\tilde{K}_1^q)\rightarrow\mathcal{L}(\tilde{K}_0^q)\rightarrow\mathcal{L}(\mathrm{S}(\mathfrak{u}^*)\otimes_\k E_1\otimes_\k E')\rightarrow0.$$ Since $\mathrm{H}^\bullet$ is an exact $\delta$-functor, for $i>m_1+\ldots+m_j$, $$\mathrm{H}^{i+1}(G/\mathbf{B},\mathcal{L}(\mathrm{S}(\mathfrak{u}^*))\otimes_\k\mathcal{L}(E_1)\otimes_\k\mathcal{L}(E'))=0\mbox{, if }\mathrm{H}^{i+1+p}(G/\mathbf{B},\mathcal{L}(\tilde{K}_p^q))=0,$$ for all nonnegative integers $q$ and $p$, but, by induction hypothesis, $$\mathrm{H}^{i+1+p}(G/\mathbf{B},\mathcal{L}(\tilde{K}_p^q))=0\mbox{, for }i>m_1+\ldots+m_j,$$ whence the Proposition.
\end{proof}

Let $\mathfrak{v}$ be a subalgebra of $\mathfrak{b}$ containing $\mathfrak{u}$. For any $\mathbf{B}$-module $E$, we denoted by $\mathcal{M}(E)$ the quotient of $G\times\mathfrak{v}\times\mathfrak{v}\times E$ under the right action of $\mathbf{B}$ given by $$(g,x,y,v).b:=(gb,b^{-1}(x),b^{-1}(y),b^{-1}(v)).$$ Then $\mathcal{M}(E)$ is a fiber bundle over $G\times_\mathbf{B}(\mathfrak{v}\times\mathfrak{v})$.

\begin{prop}\label{P2} Let $F$ be the trivial $\mathbf{B}$-module of dimension $1$. Then $$\mathrm{H}^i(G\times_\mathbf{B}(\mathfrak{v}\times\mathfrak{v}),\mathcal{M}(F))=0\mbox{, }\forall i>0.$$
\end{prop}
\begin{proof} We denote by $\mathfrak{g}^*$ and $\mathfrak{v}^*$ the duals of $\mathfrak{g}$ and $\mathfrak{v}$ respectively, and we denote by $\psi$ the canonical morphism from $G\times_\mathbf{B}(\mathfrak{v}\times\mathfrak{v})$ to $G/\mathbf{B}$. Then $\mathcal{M}(F)$ is equal to $\psi^*(\mathcal{L}(F))$. As it is easy to see that $\psi_*(\mathcal{O}_{G\times_\mathbf{B}(\mathfrak{v}\times\mathfrak{v})})$ is equal to $\mathcal{L}(\mathfrak{v}^*\times\mathfrak{v}^*)$. So we have the equality:$$\psi_*\psi^*(\mathcal{L}(F))=\psi_*(\mathcal{O}_{G\times_\mathbf{B}(\mathfrak{v}\times\mathfrak{v})}) \otimes_{\mathcal{O}_{G/\mathbf{B}}}\mathcal{L}(F)=\displaystyle\sum_{q\in\mathbb{N}} \mathcal{L}(\mathrm{S}^q(\mathfrak{v}^*\times\mathfrak{v}^*)).$$ From this equality, we deduce the equality $$\mathrm{H}^i(G\times_\mathbf{B}(\mathfrak{v}\times\mathfrak{v}),\mathcal{M}(F)) =\displaystyle\sum_{q\in\mathbb{N}} \mathrm{H}^i(G/\mathbf{B},\mathcal{L}(\mathrm{S}^q(\mathfrak{v}^*\times \mathfrak{v}^*))),\mbox{   }(*)$$ for any nonnegative integer $i$.\\
Let $E$ be a $\mathbf{B}$-module which is a direct sum of finite dimensional $\mathbf{B}$-modules and let $M_E$ be the $\mathrm{S}(\mathfrak{g}^*)$-module $$M_E:=\mathrm{S}(\mathfrak{g}^*)\otimes_\Bbbk E.$$Since $\mathfrak{g}$ and $\mathfrak{g}^*$ are identified by $\<.,.\>$, the kernel of the canonical projection from $\mathfrak{g}^*$ to $\mathfrak{v}^*$ is equal to $\mathfrak{k}$, whenever $\mathfrak{k}$ is the orthogonal complement of $\mathfrak{v}$ in $\mathfrak{g}$. Then the Koszul complex,$$\ldots\rightarrow K_2(E)\rightarrow K_1(E)\rightarrow K_0(E)\rightarrow\mathrm{S}(\mathfrak{v}^*)\otimes_\Bbbk E\rightarrow 0$$ defined by: $$K_n(E):=M_E\otimes_\Bbbk\wedge^n(\mathfrak{k})\mbox{ , }\mathrm{d}:K_{n+1}(E)\rightarrow K_n(E)\atop \mathrm{d}(m\otimes(a_0\wedge\ldots\wedge a_n))=\displaystyle\sum_{i=0}^n(-1)^ia_im\otimes(a_0\wedge\ldots\wedge\hat a_i\wedge\ldots\wedge a_n),$$is exact. This complex $K_\bullet(E)$ is canonically graded by $$K_\bullet(E):=\displaystyle\sum_{q}K_\bullet^q(E)\mbox{, where }K_n^q(E):=\mathrm{S}^{q-n}(\mathfrak{g^*})\otimes_\Bbbk E\otimes_\Bbbk \wedge^n(\mathfrak{k}).$$ Thus we have the long exact sequence of $\mathbf{B}$-modules $$\ldots\rightarrow K_2^q(E)\rightarrow K_1^q(E)\rightarrow K_0^q(E)\rightarrow\mathrm{S}^q(\mathfrak{v}^*)\otimes_\Bbbk E\rightarrow 0.$$ Since $\mathrm{H}^\bullet$ is an exact $\delta$-functor, for $i$ nonnegative integer, $$\mathrm{H}^i(G/\mathbf{B},\mathcal{L}(\mathrm{S}(\mathfrak{v}^*)\otimes_\Bbbk E))=0\mbox{, if }\mathrm{H}^{i+n}(G/\mathbf{B},\mathcal{L}(K_n^q(E)))=0,$$ for any nonnegative integers $q$ and $n$.

Since $\mathrm{S}^{q-n}(\mathfrak{g}^*)$ is a $G$-module, the $\mathcal{O}_{G/\mathbf{B}}$-modules $\mathcal{L}(\mathrm{S}^{q-n}(\mathfrak{g}^*)\otimes_\Bbbk E\otimes_\Bbbk\wedge^n(\mathfrak{k}))$ and $\mathrm{S}^{q-n}(\mathfrak{g}^*)\otimes_\Bbbk\mathcal{L}(E\otimes_\Bbbk\wedge^n(\mathfrak{u}))$ are isomorphic. Consequently, $$\mathrm{H}^{i+n}(G/\mathbf{B},\mathcal{L}(K_n^q(E)))=\mathrm{S}^{q-n}(\mathfrak{g}^*)\otimes_\Bbbk\mathrm{H}^{i+n}(G/\mathbf{B},\mathcal{L}(E\otimes_\Bbbk\wedge^n(\mathfrak{k}))).$$ As a result, $$\mathrm{H}^i(G/\mathbf{B},\mathcal{L}(\mathrm{S}(\mathfrak{v}^*)\otimes_\Bbbk E))=0\mbox{, if }\mathrm{H}^{i+n}(G/\mathbf{B},\mathcal{L}(E\otimes_\Bbbk\wedge^n(\mathfrak{k})))=0,$$ for any nonnegative integer $n$.

Let us consider for $\mathbf{B}$-module $E$ the $\mathbf{B}$-module $\mathrm{S}(\mathfrak{v}^*)$. Then by which goes before, $$\mathrm{H}^i(G/\mathbf{B},\mathcal{L}(\mathrm{S}(\mathfrak{v}^*)\otimes\mathrm{S}(\mathfrak{v}^*)))=0\mbox{, if }\mathrm{H}^{i+n+m}(G/\mathbf{B},\mathcal{L}(\wedge^n(\mathfrak{k})\otimes_\Bbbk\wedge^m(\mathfrak{k})))=0,$$ for any nonnegative integers $n$ and $m$. By Proposition \ref{PC}, $$\mathrm{H}^{i+n+m}(G/\mathbf{B},\mathcal{L}(\wedge^n(\mathfrak{k})\otimes_\Bbbk\wedge^m(\mathfrak{k})))=0\mbox{, }\forall i>0,$$ then by $(*)$, $$\mathrm{H}^i(G\times_\mathbf{B}(\mathfrak{v}\times\mathfrak{v}),\mathcal{M}(F))=0\mbox{, }\forall i>0.$$ 
\end{proof}

\begin{theo} We have the equalities:$$(\vartheta_\n)_*(\mathcal{O}_{G\times_\mathbf{B}(\mathfrak{u}\times\mathfrak{u})})=\mathcal{O}_{\mathcal{N}_\n}\mbox{, }(\gamma_\n)_*(\mathcal{O}_{G\times_\mathbf{B}(\mathfrak{b}\times\mathfrak{b})})=\mathcal{O}_{{\mathcal B}_{\mathfrak{g},\n}}$$ and  for $i>0$, $$\mathrm{H}^i(G\times_\mathbf{B}(\mathfrak{u}\times\mathfrak{u}),\mathcal{O}_{G\times_\mathbf{B}(\mathfrak{u}\times\mathfrak{u})}) =0\mbox{ and }\mathrm{H}^i(G\times_\mathbf{B}(\mathfrak{b}\times\mathfrak{b}),\mathcal{O}_{G\times_\mathbf{B}(\mathfrak{b}\times\mathfrak{b})}) =0.$$
\end{theo}
\begin{proof} Since $\mathcal{N}_\n$ and ${\mathcal B}_{\mathfrak{g},\n}$ are normal, $$(\vartheta_\n)_*(\mathcal{O}_{G\times_\mathbf{B}(\mathfrak{u}\times\mathfrak{u})})=\mathcal{O}_{\mathcal{N}_\n}\mbox{ and }(\gamma_\n)_*(\mathcal{O}_{G\times_\mathbf{B} (\mathfrak{b}\times\mathfrak{b})})=\mathcal{O}_{{\mathcal{B}_{\mathfrak{g},\n}}}.$$ By Proposition \ref{P2}, $$\mathrm{H}^i(G\times_\mathbf{B}(\mathfrak{u}\times\mathfrak{u}),\mathcal{O}_{G\times_\mathbf{B}(\mathfrak{u}\times\mathfrak{u})}) =\mathrm{H}^i(G\times_\mathbf{B}(\mathfrak{u}\times\mathfrak{u}),\mathcal{M}(E(0)))=0\mbox{, for }i>0$$ since $\mathcal{O}_{G\times_\mathbf{B}(\mathfrak{u}\times\mathfrak{u})}=\psi^*(\mathcal{O}_{G/\mathbf{B}})=\psi^*(\mathcal{L}(E(0)))$, and $$\mathrm{H}^i(G\times_\mathbf{B}(\mathfrak{b}\times\mathfrak{b}),\mathcal{O}_{G\times_\mathbf{B}(\mathfrak{b}\times\mathfrak{b})}) =\mathrm{H}^i(G\times_\mathbf{B}(\mathfrak{b}\times\mathfrak{b}),\mathcal{M}(E(0)))=0\mbox{, for }i>0$$ since $\mathcal{O}_{G\times_\mathbf{B}(\mathfrak{b}\times\mathfrak{b})}=\psi^*(\mathcal{O}_{G/\mathbf{B}})=\psi^*(\mathcal{L}(E(0)))$.
\end{proof}

Since $\vartheta_\n$ and $\gamma_\n$ are desingularisation morphisms of $\mathcal{N}_\n$ and ${\mathcal B}_{\mathfrak{g},\n}$ respectively, the corollary is a consequence of the theorem.

\begin{cor} The normalizations $\mathcal{N}_\n$ and ${\mathcal B}_{\mathfrak{g},\n}$ of $\mathcal{N}$ and $\mathcal{B}_\mathfrak{g}$ respectively have rational singularities.
\end{cor}

\appendix 
\section{ }
Recall that $\mathfrak{g}$ is simple and $\mathcal{R}$, $\mathcal{R}_+$, $\Pi$ are the root data. We denote by $\beta_1,\ldots,\beta_{\rk\mathfrak{g}}$ the elements of $\Pi$ ordered as in \cite{3} and by $\rho$ the half sum of positive roots. In particular, if $\mathfrak{g}$ has not type $D_{\rk\mathfrak{g}}$ or $\E$, $\beta_i$ is not orthogonal to $\beta_{i+1}$ for $i=1,\ldots,\rk\mathfrak{g}-1$. When $\mathfrak{g}$ has type $\D_{\rk\mathfrak{g}}$, $\beta_{\rk\mathfrak{g}-1}$ is orthogonal to $\beta_{\rk\mathfrak{g}}$. When $\mathfrak{g}$ has type $\E$, $\beta_2$ is not orthogonal to $\beta_4$ and $\beta_3$ is not orthogonal to $\beta_1$ and $\beta_4$. For $\beta$ in $\Pi$, the reflexion associated to $\beta$ is denoted by $s_\beta$.\\

\begin{Lemme}\label{L6.1} We suppose $\mathfrak{g}$ simple. Let $\alpha$ be in $\mathcal{R}_+$.
\begin{itemize}
\item[(i)] If $\alpha$ is simple, then $s_\alpha(\rho-\alpha)$ is regular dominant.
\item[(ii)] If $\alpha$ is not simple, then $\rho-\alpha$ is not regular.
\item[(iii)] If $\alpha$ is the biggest root then $\rho+\alpha$ is regular dominant.
\item[(iv)] If $\alpha$ is not the biggest root then $\rho+\alpha$ is not regular when $\mathfrak{g}$ has type $\A_{\rk\mathfrak{g}}$, $\E_6$, $\E_7$, $\E_8$.
\item[(v)] If $\alpha$ is not the biggest root:
\begin{itemize}
\item[(a)] for $\mathfrak{g}$ of type $\B_{\rk\mathfrak{g}}$, $\F_4$, $\G_2$, $\rho+\alpha$ is regular for two values of $\alpha$ and when it is not dominant there exists a simple root $\beta$ such that $s_\beta(\rho+\alpha)$ is dominant,
\item[(b)] for $\mathfrak{g}$ of type $\C_{\rk\mathfrak{g}}$, $\D_{\rk\mathfrak{g}}$, $\rho+\alpha$ is regular for one value of $\alpha$ and for this value it is dominant.
\end{itemize}
 \end{itemize}
\end{Lemme}
\begin{proof} (i) Let $\alpha$ be a simple root. Then $s_\alpha(\rho)=\rho-\alpha$. So $s_\alpha(\rho-\alpha)$ is regular dominant since $\rho$ is regular dominant.\\

(iii) Let $\alpha$ be the biggest root. Then for any positive root $\gamma$, $\<\alpha,\gamma\>$ is nonnegative. Consequently, $\rho+\alpha$ is regular and dominant since  $\<\rho,\beta\>$ is positive for any $\beta$ in $\Pi$.\\ 

(ii) {\it If $\alpha$ is not simple, then $\rho-\alpha$ is not regular.}\\

If $\beta_{i_1},\ldots,\beta_{i_l}$ are pairwise different elements of $\Pi$ such that $\beta_{i_j}$ is not orthogonal to $\beta_{i_{j+1}}$ for $j=1,\ldots,l-1$ and so that$$s_{\beta_{i_l}}(\beta_{i_{l-1}})=\beta_{i_{l-1}}+\beta_{i_l},$$ then$$s_{\beta_{i_l}}(\rho-(\beta_{i_1}+\cdots+\beta_{i_l}))=\rho-(\beta_{i_1}+\cdots+\beta_{i_l}),$$ since $s_{\beta_{i_l}}(\rho)=\rho-\beta_{i_l}$. Hence $\rho-\alpha$ is not regular if $\alpha$ is the sum of pairwise different simple roots since we can order them so that the last condition is satisfied. We denote by $\mathcal{R}'_+$ the subset of positive roots whose coordinates in the basis $\Pi$ are not all at most to 1. Then it remains to prove that $\rho-\alpha$ is not regular for any $\alpha$ in $\mathcal{R}'_+$. For that purpose, it will be enough to find $\beta$ in $\Pi$ such that $s_\beta(\rho-\alpha)=\rho-\alpha$. We then consider the different possible cases.\\

1) $\mathfrak{g}$ has type $\B_{\rk\mathfrak{g}}$ and $$\alpha=\beta_i+\cdots+\beta_k+2(\beta_{k+1}+\cdots+ \beta_{\rk\mathfrak{g}}),$$ for $1\leqslant i\leqslant k<\rk\mathfrak{g}$. If $i<k$, then $$s_{\beta_i}(\rho-\alpha)=\rho-\alpha.$$ If $i=k<\rk\mathfrak{g}-1$ then $$s_{\beta_{i+1}}(\rho-\alpha)=\rho-\beta_{i+1}-(\beta_i+\beta_{i+1})+2\beta_{i+1}-2(\beta_{i+1}+\beta_{i+2}) \atop-(\alpha-\beta_i-2\beta_{i+1}-2\beta_{i+2})=\rho-\alpha.$$ If $\alpha=\beta_{\rk\mathfrak{g}-1}+2 \beta_{\rk\mathfrak{g}}$, then $$s_{\beta_{\rk\mathfrak{g}}}(\rho-\alpha)=\rho-\beta_{\rk\mathfrak{g}}- (\beta_{\rk\mathfrak{g}-1}+2\beta_{\rk\mathfrak{g}})+2\beta_{\rk\mathfrak{g}}\atop=\rho-\beta_{\rk\mathfrak{g}-1}-
\beta_{\rk\mathfrak{g}}.$$ In this case $\rho-\alpha$ is not regular since $\rho-\beta_{\rk\mathfrak{g}-1}- \beta_{\rk\mathfrak{g}}$ is not regular by which goes before.\\

2) $\mathfrak{g}$ has type $\C_{\rk\mathfrak{g}}$. Let us suppose$$\alpha=\beta_i+\cdots+\beta_k+2(\beta_{k+1}+\cdots+ \beta_{\rk\mathfrak{g}-1})+\beta_{\rk\mathfrak{g}},$$for $1\leqslant i\leqslant k<\rk\mathfrak{g}$. If $i<k$, then $$s_{\beta_i}(\rho-\alpha)=\rho-\alpha.$$  If $i=k$ then $$s_{\beta_{i+1}}(\rho-\alpha)=\rho-\beta_{i+1}-(\beta_i+\beta_{i+1})+2\beta_{i+1}-2(\beta_{i+1}+\beta_{i+2}) \atop-(\alpha-\beta_i-2\beta_{i+1}-2\beta_{\i+2})=\rho-\alpha.$$ Let us suppose $$\alpha=2(\beta_k+\cdots+ \beta_{\rk\mathfrak{g}-1})+\beta_{\rk\mathfrak{g}},$$ for $1\leqslant k<\rk\mathfrak{g}$. If $k<\rk\mathfrak{g}-1$, then $$s_{\beta_k}(\rho-\alpha)=\rho-\beta_k-2(\beta_{k+1}+\cdots+ \beta_{\rk\mathfrak{g}-1})-\beta_{\rk\mathfrak{g}}.$$ Hence $\rho-\alpha$ is not regular since the right hand side of the last equality is not regular. If $k=\rk\mathfrak{g}-1$, then $$s_{\beta_{\rk\mathfrak{g}-1}}(\rho-\alpha)=\rho-\beta_{\rk\mathfrak{g}-1}-\beta_{\rk\mathfrak{g}}.$$ Hence $\rho-!
 \alpha$ is not regular since $\rho-\beta_{\rk\mathfrak{g}-1}-\beta_{\rk\mathfrak{g}}$ is not regular.\\

3) $\mathfrak{g}$ has type $\D_{\rk\mathfrak{g}}$. Let us suppose$$\alpha=\beta_i+\ldots+\beta_k+2(\beta_{k+1}+\ldots+ \beta_{\rk\mathfrak{g}-2})+\beta_{\rk\mathfrak{g}}+\beta_{\rk\mathfrak{g}-1},$$for $1\leqslant i\leqslant k<\rk\mathfrak{g}-2$. If $i<k$, then $$s_{\beta_i}(\rho-\alpha)=\rho-\alpha.$$ If $i=k<\rk\mathfrak{g}-3$ then $$s_{\beta_{i+1}}(\rho-\alpha)=\rho-\beta_{i+1}-(\beta_i+\beta_{i+1})+2\beta_{i+1}-2(\beta_{i+1}+\beta_{i+2}) \atop-(\alpha-\beta_i-2\beta_{i+1}-2\beta_{i+2})=\rho-\alpha.$$ If $i=k=\rk\mathfrak{g}-3$, then $$s_{\beta_{\rk\mathfrak{g}-2}}(\rho-\alpha)=\rho-\beta_{\rk\mathfrak{g}-2}-(\beta_{\rk\mathfrak{g}-3} +\beta_{\rk\mathfrak{g}-2})+2\beta_{\rk\mathfrak{g}-2}\atop-(\beta_{\rk\mathfrak{g}-1}+\beta_{\rk\mathfrak{g}-2}) -(\beta_{\rk\mathfrak{g}}+\beta_{\rk\mathfrak{g}-2})=\rho-\alpha.$$ So in any case, $\rho-\alpha$ is not regular.\\

4) $\mathfrak{g}$ has type $\E$. For $i=1,\ldots,\rk\mathfrak{g}$, we denote by $J_i$ the subset of $j$ in $\{1,\ldots,\rk\mathfrak{g}\}\backslash\{i\}$ such that $\beta_j$ is not orthogonal to $\beta_i$. Then $|J_i|\leqslant2$ if $i\neq4$ and $|J_4|=3$. For $i=1,\ldots,\rk\mathfrak{g}$, we denote by $n_i$ the coordinate of $\alpha$ at $\beta_i$ in the basis $\Pi$ and we set $$L(\alpha):=[n_2,n_1,n_3,n_4,\ldots,n_{\rk\mathfrak{g}}].$$ If there exists some $i$ such that $$2n_i-1=\sum\limits_{j\in J_i}n_j,\mbox{    }(4)$$ then $$s_{\beta_i}(\rho-\alpha)=\rho-\beta_i+n_i\beta_i-s_{\beta_i}(\sum\limits_{j\in J_i}n_j\beta_j) -\sum\limits_{j\notin J_i\cup\{i\}}n_j\beta_j\atop=\rho+(n_i-1)\beta_i-(\sum\limits_{j\in J_i}n_j)\beta_i-\alpha+n_i\beta_i =\rho-\alpha,$$ since all the roots have the same lenght. Consequently, if $\alpha$ satisfies the equality $(4)$ for some $i$, then $\rho-\alpha$ is not regular. For each case, we will give some $i$ for which $\alpha$ satisfies equality $(4)$. The result will be presented in a table. Its first column gives some values of $i$ and the numbers $j$ wich are on the same line in the second column are such that the root $r(j)$ satisfies equality $(4)$ with respect to $i$. The roots are given by the map $\alpha\mapsto L(\alpha)$.\\
a) We suppose $\mathfrak{g}$ of type $\E_6$. The image of $\mathcal{R}'_+$ by the map $\alpha\mapsto L(\alpha)$ is given by
$$\begin{tabular}{lcc}
r(1):=[1,0,1,2,1,0]&r(2):=[1,1,1,2,1,0]&r(3):=[1,0,1,2,1,1]\\
r(4):=[1,1,2,2,1,0]&r(5):=[1,1,1,2,1,1]&r(6):=[1,0,1,2,2,1]\\
r(7):=[1,1,2,2,1,1]&r(8):=[1,1,1,2,2,1]&r(9):=[1,1,2,2,2,1]\\
r(10):=[1,1,2,3,2,1]&r(11):=[2,1,2,3,2,1].\\
\end{tabular}$$
The corresponding table is given by 
$$\begin{tabular}{lc}
2&11\\
3&4\\
4&1,2,10\\
5&6,8,9\\
6&3,5,7\\
\end{tabular}$$

b) We suppose that $\mathfrak{g}$ has type $\E_7$. Let $\mathcal{R}''_+$ be the subset of elements of $\mathcal{R}'_+$ such that $n_7\neq0$. The image of $\mathcal{R}''_+$ by the map $\alpha\mapsto L(\alpha)$ is given by
$$\begin{tabular}{lcc}
r(1):=[1,0,1,2,1,1,1]&r(2):=[1,1,1,2,1,1,1]&r(3):=[1,0,1,2,2,1,1]\\
r(4):=[1,1,2,2,1,1,1]&r(5):=[1,1,1,2,2,1,1]&r(6):=[1,0,1,2,2,2,1]\\
r(7):=[1,1,2,2,2,1,1]&r(8):=[1,1,1,2,2,2,1]&r(9):=[1,1,2,2,2,2,1]\\
r(10):=[1,1,2,3,2,1,1]&r(11):=[1,1,2,3,2,2,1]&r(12):=[2,1,2,3,2,1,1]\\
r(13):=[1,1,2,3,3,2,1]&r(14):=[2,1,2,3,2,2,1]&r(15):=[2,1,2,3,3,2,1]\\
r(16):=[2,1,2,4,3,2,1]&r(17):=[2,1,3,4,3,2,1]&r(18):=[2,2,3,4,3,2,1].\\
\end{tabular}$$
The corresponding table is given by
$$\begin{tabular}{lc}
1&18\\
3&17\\
4&16\\
5&13,15\\
6&6,8,9,11,14\\
7&1,2,3,4,5,7,10,12\\
\end{tabular}$$

c) We suppose that $\mathfrak{g}$ has type $\E_8$. Let $\mathcal{R}''_+$ be the subset of elements of $\mathcal{R}'_+$ such that $n_8\neq0$. The image of $\mathcal{R}''_+$ by the map $\alpha\mapsto L(\alpha)$ is given by
$$\begin{tabular}{lcc}
r(1):=[1,0,1,2,1,1,1,1]&r(2):=[1,0,1,2,2,1,1,1]&r(3):=[1,1,1,2,1,1,1,1]\\
r(4):=[1,0,1,2,2,2,1,1]&r(5):=[1,1,2,2,1,1,1,1]&r(6):=[1,1,1,2,2,1,1,1]\\
r(7):=[1,1,2,2,2,1,1,1]&r(8):=[1,1,1,2,2,2,1,1]&r(9):=[1,0,1,2,2,2,2,1]\\
r(10):=[1,1,2,3,2,1,1,1]&r(11):=[1,1,2,2,2,2,1,1]&r(12):=[1,1,1,2,2,2,2,1]\\
r(13):=[2,1,2,3,2,1,1,1]&r(14):=[1,1,2,3,2,2,1,1]&r(15):=[1,1,2,2,2,2,2,1]\\
r(16):=[2,1,2,3,2,2,1,1]&r(17):=[1,1,2,3,3,2,1,1]&r(18):=[1,1,2,3,2,2,2,1]\\
r(19):=[2,1,2,3,3,2,1,1]&r(20):=[2,1,2,3,2,2,2,1]&r(21):=[1,1,2,3,3,2,2,1]\\
r(22):=[2,1,2,4,3,2,1,1]&r(23):=[2,1,2,3,3,2,2,1]&r(24):=[1,1,2,3,3,3,2,1]\\
r(25):=[2,1,3,4,3,2,1,1]&r(26):=[2,1,2,4,3,2,2,1]&r(27):=[2,1,2,3,3,3,2,1]\\
r(28):=[2,2,3,4,3,2,1,1]&r(29):=[2,1,3,4,3,2,2,1]&r(30):=[2,1,2,4,3,3,2,1]\\
r(31):=[2,2,3,4,3,2,2,1]&r(32):=[2,1,3,4,3,3,2,1]&r(33):=[2,1,2,4,4,3,2,1]\\
r(34):=[2,2,3,4,3,3,2,1]&r(35):=[2,1,3,4,4,3,2,1]&r(36):=[2,1,3,5,4,3,2,1]\\
r(37):=[2,2,3,4,4,3,2,1]&r(38):=[3,1,3,5,4,3,2,1]&r(39):=[2,2,3,5,4,3,2,1]\\
r(40):=[3,2,3,5,4,3,2,1]&r(41):=[2,2,4,5,4,3,2,1]&r(42):=[3,2,4,5,4,3,2,1]\\
r(43):=[3,2,4,6,4,3,2,1]&r(44):=[3,2,4,6,5,3,2,1]&r(45):=[3,2,4,6,5,4,2,1]\\
r(46):=[3,2,4,6,5,4,3,1]&r(47):=[3,2,4,6,5,4,3,2].\\
\end{tabular}$$
The corresponding table is given by
$$\begin{tabular}{lc}
1&\\
2&38,40\\
3&41,42\\
4&36,39,43\\
5&33,35,37,44\\
6&24,27,30,32,34,45\\
7&9,12,15,18,20,21,23,26,29,31,46\\
8&1..8,10,11,13,14,16,17,19,22,25,28,47\\
\end{tabular}$$

5) $\mathfrak{g}$ has type $\F_4$. For $i=1,\ldots,4$, we denote by $n_i$ the coordinate of $\alpha$ at $\beta_i$ in the basis $\Pi$ and we set $$L(\alpha):=[n_1,n_2,n_3,n_4].$$ We identify $\alpha$ with $L(\alpha)$. If $2n_2+n_4=2n_3-1$, then $$s_{\beta_3}(\rho-\alpha)=\rho-\beta_3-n_1\beta_1-n_2\beta_2-n_4\beta_4-(2n_2+n_4-n_3)\beta_3\atop=\rho-\alpha.$$ If $n_1+n_3=2n_2-1$, then $$s_{\beta_2}(\rho-\alpha)=\rho-\beta_2-n_1\beta_1-n_3\beta_3-n_4\beta_4-(n_1+n_3-n_2)\beta_2\atop=\rho-\alpha.$$ If $n_2=2n_1-1$, then $$s_{\beta_1}(\rho-\alpha)=\rho-\beta_1-n_2\beta_2-n_3\beta_3-n_4\beta_4-(-n_1+n_2)\beta_1\atop=\rho-\alpha.$$ Analogously, $s_{\beta_4}(\rho-\alpha)=\rho-\alpha$ if $n_3=2n_4-1$. The image of $\mathcal{R}'_+$ by the map $\alpha\mapsto L(\alpha)$ is given by
$$\begin{tabular}{lcc}
r(1):=[0,1,2,0]&r(2):=[1,1,2,0]&r(3):=[0,1,2,1]\\
r(4):=[1,2,2,0]&r(5):=[1,1,2,1]&r(6):=[0,1,2,2]\\
r(7):=[1,2,2,1]&r(8):=[1,1,2,2]&r(9):=[1,2,3,1]\\
r(10):=[1,2,2,2]&r(11):=[1,2,3,2]&r(12):=[1,2,4,2]\\
r(13):=[1,3,4,2]&r(14):=[2,3,4,2].\\
\end{tabular}$$
We then deduce the following table
$$\begin{tabular}{lc}
1&2,5,8,14\\
2&4,7,10,13\\
3&3,5,9\\
4&11\\
\end{tabular}$$

From the equalities:$$s_{\beta_3}(\rho-r(1))=\rho-\beta_3-\beta_2$$ $$s_{\beta_4}(\rho-r(6))=\rho-\beta_4-\beta_2-2\beta_3\atop= \rho-r(3)$$ $$s_{\beta_3}(\rho-r(12))=\rho-\beta_3-\beta_1-2(\beta_2+2\beta_3)+4\beta_3-2\beta_3-2\beta_4\atop= \rho-r(11).$$  So for $i=1,6,7$, $\rho-r(i)$ is not regular by which goes before.\\

6)  $\mathfrak{g}$ has type $\G_2$. From the equalities:$$s_{\beta_1}(\rho-(2\beta_1+\beta_2))=\rho-\beta_1+2\beta_1- (\beta_2+3\beta_1)\atop=\rho-(\beta_2+2\beta_1)$$ $$s_{\beta_2}(\rho-(3\beta_1+2\beta_2))=\rho-\beta_2+2\beta_2- 3(\beta_2+\beta_1)\atop=\rho-(3\beta_1+2\beta_2)$$ $$s_{\beta_1}(\rho-(3\beta_1+\beta_2))=\rho-\beta_1+3\beta_1- (\beta_2+3\beta_1)\atop=\rho-(\beta_2+\beta_1).$$ So for any $\alpha$ in $\mathcal{R}'_+$, $\rho-\alpha$ is not regular.\\

(iv) {\it If $\alpha$ is not the biggest root then $\rho+\alpha$ is not regular when $\mathfrak{g}$ has type $\A_{\rk\mathfrak{g}}$, $\E_6$, $\E_7$, $\E_8$.}\\

Let us suppose that $\alpha$ is not the biggest root and that $\mathfrak{g}$ has type $\A_{\rk_\mathfrak{g}}$, $\E_6$, $\E_7$, $\E_8$. If $\beta_{i_1},\ldots,\beta_{i_l}$ are pairwise different elements of $\Pi$ such that $\beta_{i_j}$ is not orthogonal to $\beta_{i_{j+1}}$ for $j=1,\ldots,l-1$, then $$s_{\beta_{i_l}}(\rho+\beta_{i_1}+\cdots+\beta_{i_{l-1}}) =\rho+\beta_{i_1}+\cdots+\beta_{i_{l-1}},$$ since $s_{\beta_{i_l}}(\rho)=\rho-\beta_{i_l}$. Hence $\rho+\alpha$ is not regular if $\alpha$ is the sum of $l<\rk\mathfrak{g}$ pairwise different simple roots. In particular, the statement is proved when $\mathfrak{g}$ has type $\A_{\rk\mathfrak{g}}$. So we can suppose that $\mathfrak{g}$ has type $\E$. We then use the notations of (ii,4). If there exists some $i$ such that $$2n_i+1=\sum\limits_{j\in J_i}n_j,\mbox{   }(5)$$ then $$s_{\beta_i}(\rho+\alpha)=\rho-\beta_i-n_i\beta_i+ s_{\beta_i}(\sum\limits_{j\in J_i}n_j\beta_j) +\sum\limits_{j\notin J_i\cup\{i\}}n_j\beta_j\atop=!
 \rho-(n_i+1)\beta_i+(\sum\limits_{j\in J_i}n_j)\beta_i+\alpha-n_i\beta_i =\rho+\alpha,$$ since all the roots have the same length. Consequently, if $\alpha$ satisfies the equality $(5)$ for some $i$, then $\rho+\alpha$ is not regular. If all the coordinates of $\alpha$ are equal to $1$, then $\alpha$ satisfies equality $(5)$ for $i=4$. Hence $\rho+\alpha$ is not regular in this case and we have only to consider the cases when $\alpha$ belongs to $\mathcal{R}'_+$. For each case, we will give some $i$ for which $\alpha$ satisfies equality $(5)$. The result will be presented in a table. Its first column gives some values of $i$ and the numbers $j$ which are on the same line in the second column are such that the root $r(j)$ satisfies equality $(5)$ with respect $i$. The roots are given by their images by the map $\alpha\mapsto L(\alpha)$. We recall that the biggest root corresponds to $r(i)$ for $i$ the biggest one.\\
1) $\mathfrak{g}$ has type $\E_6$. The corresponding table is given by
$$\begin{tabular}{lc}
1&6\\
2&10\\
3&8\\
4&9\\
5&3,5,7\\
6&1,2,4\\
\end{tabular}$$

2) $\mathfrak{g}$ has type $\E_7$. Let $\mathcal{R}''_+$ be the subset of elements of $\mathcal{R}'_+$ such that $n_7\neq0$. Any element of $\mathcal{R}'_+\backslash\mathcal{R}''_+$ which satisfies equality $(5)$ with respect to $i\leqslant5$ and $\E_6$ satisfy the same equality with respect to $i$ and $\mathfrak{g}$ since $\beta_i$ is orthogonal to $\beta_7$. If it satisfies equality $(5)$ with respect to $6$ and $\E_6$, then it  satisfies the same equality with respect to $i$ and $\mathfrak{g}$ since $n_7=0$. If $\alpha$ is the biggest root of the subsystem generated by $\beta_1,\ldots,\beta_6$, then $n_6=1$. So $s_{\beta_7}(\rho+\alpha)=\rho+\alpha$ since $s_{\beta_7}(\beta_6)=\beta_6+\beta_7$. Hence we have only to consider the elements of $\mathcal{R}''_+$. The corresponding table is given by
$$\begin{tabular}{lc}
1&6,17\\
2&13\\
3&8,16\\
4&9,15\\
5&1,2,4,11,14\\
6&3,5,7,10,12\\
\end{tabular}$$

3) $\mathfrak{g}$ has type $\E_8$. Let $\mathcal{R}''_+$ be the subset of elements of $\mathcal{R}'_+$ such that $n_8\neq0$. Any element of $\mathcal{R}'_+\backslash\mathcal{R}''_+$ which satisfies equality $(5)$ with respect to $i\leqslant6$ and $\E_7$ satisfy the same equality with respect to $i$ and $\mathfrak{g}$ since $\beta_i$ is orthogonal to $\beta_8$. If it satisfies equality $(5)$ with respect to $7$ and $\E_7$, then it  satisfies the same equality with respect to $i$ and $\mathfrak{g}$ since $n_8=0$. If $\alpha$ is the biggest root of the subsystem generated by $\beta_1,\ldots,\beta_7$, then $n_7=1$. So $s_{\beta_8}(\rho+\alpha)=\rho+\alpha$ since $s_{\beta_8}(\beta_7)=\beta_7+\beta_8$. Hence we have only to consider the elements of $\mathcal{R}''_+$. The corresponding table is given by
$$\begin{tabular}{lc}
1&38\\
2&24,36,41\\
3&9,12,33,39,40\\
4&15,27,35,37,42\\
5&1,3,5,18,20,30,32,34,43\\
6&2,6,7,10,13,21,23,26,29,31,44\\
7&4,8,11,14,16,17,19,22,25,28,45\\
8&16\\
\end{tabular}$$

(v) {\it If $\alpha$ is not the biggest root:
\begin{itemize}
\item[(a)] for $\mathfrak{g}$ of type $\B_{\rk\mathfrak{g}}$, $\F_4$, $\G_2$, $\rho+\alpha$ is regular for two values of $\alpha$ and when it isn't dominant there exists a simple root $\beta$ such that $s_\beta(\rho+\alpha)$ is dominant,
\item[(b)] for $\mathfrak{g}$ of type $\C_{\rk\mathfrak{g}}$, $\D_{\rk\mathfrak{g}}$, $\rho+\alpha$ is regular for one value of $\alpha$ and for this value it is dominant.
\end{itemize}}

(a) 1)  For $\mathfrak{g}$ of type $\B_{\rk\mathfrak{g}}$, let be $\alpha_1$ be the sum of simple roots. For any positive root $\gamma$, $\<\alpha_1,\gamma\>$ is nonnegative, then $\alpha_1$ is dominant. Hence $\rho+\alpha_1$ is regular dominant. Let be $\alpha_2$ be the sum of the simple roots of biggest length. Then $\alpha_2=\beta_1+\ldots+\beta_{\rk\mathfrak{g}-1}$. Since $s_{\beta_{\rk\mathfrak{g}}}(\rho+\alpha_2)=\rho+\alpha_1$, $\rho+\alpha_2$ is regular.
It remains to show that $\rho+\alpha$ is not regular for the other cases , i.e there exists a simple root $\beta$ such that $s_\beta(\rho+\alpha)=\rho+\alpha$.\\The possible values of $\alpha$ are
$$\begin{aligned}
\alpha_{1,i}=\beta_i+\cdots+\beta_{\rk\mathfrak{g}}, & 1<i\leqslant\rk\mathfrak{g}\\ 
\alpha_{2,i,j}=\beta_i+\cdots+\beta_{j-1}, & 1\leqslant i<j<\rk\mathfrak{g}\mbox{ or }1< i<j\leqslant\rk\mathfrak{g}\\ \alpha_{3,i,j}=\beta_i +\cdots+\beta_{j-1}+2(\beta_j+\cdots+\beta_{\rk\mathfrak{g}}), & 1< i<j\leqslant\rk\mathfrak{g}\\ \alpha_{4,j}=\beta_1+\cdots+\beta_j+2(\beta_{j+1}+\cdots+\beta_{\rk\mathfrak{g}}), &  1 <j<\rk\mathfrak{g}.
\end{aligned}$$
If $\alpha$ is equal to $\alpha_{1,i}$ or $\alpha_{3,i,j}$, $$s_{\beta_{i-1}}(\rho+\alpha)=\rho+\alpha,$$ if $\alpha$ is equal to $\alpha_{2,i,j}$, for $1\leqslant i<j<\rk\mathfrak{g}$, $$s_{\beta_j}(\rho+\alpha)=\rho+\alpha,$$ for $1<i<j\leqslant\rk\mathfrak{g}$, $$s_{\beta_{i-1}}(\rho+\alpha)=\rho+\alpha,$$ if $\alpha$ is equal to $\alpha_{4,j}$, $$s_{\beta_j}(\rho+\alpha)=\rho+\alpha.$$\\
2) For $\mathfrak{g}$ of type $\F_4$, let be $\alpha_1=\beta_1+2\beta_2+3\beta_3+2\beta_4$. For any positive root $\gamma$, $\<\alpha_1,\gamma\>$ is nonnegative, then $\alpha_1$ is dominant. Hence $\rho+\alpha_1$ is regular dominant. Let be $\alpha_2 =\beta_1+2(\beta_2+\beta_3+\beta_4)$. Since $s_{\beta_3}(\rho+\alpha_2)=\rho+\alpha_1$, $\rho+\alpha_2$ is regular.\\
For $i=1,2,3,4$, we denote $n_i$ the coordinate of $\alpha$ at $\beta_i$ in the basis $\Pi$ and we set $$L(\alpha):=[n_1,n_2,n_3,n_4].$$ Let $\mathcal{R}'_+$ be the set of $\alpha$ in $\mathcal{R}_+$ such that $\alpha$ is not the biggest root and it is not in $\{\alpha_1,\alpha_2\}$. The image of $\mathcal{R}'_+$ by the map $\alpha\mapsto L(\alpha)$ is given by $$\begin{tabular}{lcc}
r(1):=[1,0,0,0]&r(2):=[0,1,0,0]&r(3):=[0,0,1,0]\\
r(4):=[0,0,0,1]&r(5):=[1,1,0,0]&r(6):=[0,1,1,0]\\
r(7):=[0,0,1,1]&r(8):=[1,1,1,0]&r(9):=[0,1,1,1]\\
r(10):=[1,1,1,1]&r(11):=[0,1,2,0]&r(12):=[1,1,2,0]\\
r(13):=[0,1,2,1]&r(14):=[1,2,2,0]&r(15):=[1,1,2,1]\\
r(16):=[0,1,2,2]&r(17):=[1,2,2,1]&r(18):=[1,1,2,2]\\
r(19):=[1,2,3,1]&r(20):=[1,2,4,2]&r(21):=[1,3,4,2].
\end{tabular}$$ 
We identify $\alpha$ with $L(\alpha)$. In the following table, the value $i$ in the first column and the values $j$ on the same line verify the equality $s_{\beta_i}(\rho+r(j))=\rho+r(j)$,
$$\begin{tabular}{lc}
1&2,6,9,11,13,16,21\\
2&1,7,12,15,18,20\\
3&4,10,17\\
4&3,8,19.
\end{tabular}$$
Since $$s_{\beta_3}(\rho+r(5))=\rho+r(8)\atop s_{\beta_4}(\rho+r(14)=\rho+r(17),$$ for $i=5,14$, $\rho+r(i)$ is not regular by which goes before.\\
 
3) For $\mathfrak{g}$ of type $\G_2$, let be $\alpha_1=2\beta_1+\beta_2$. For any positive root $\gamma$, $\<\alpha_1,\gamma\>$ is nonnegative, then $\alpha_1$ is dominant. Hence $\rho+\alpha_1$ is regular dominant. Since $s_{\beta_1}(\rho+\beta_2)=\rho+\alpha_1$, $\rho+\beta_2$ is regular.\\
For $\alpha$ not in $\{\alpha_1,\beta_2\}$, $\alpha$ is in $\{\beta_1, \beta_1+\beta_2, 3\beta_1+\beta_2\}$. Since 
$$\begin{aligned}
s_{\beta_2}(\rho+\beta_1)&=\rho+\beta_1,\\
s_{\beta_1}(\rho+\beta_1+\beta_2)&=\rho+\beta_1+\beta_2,\\
s_{\beta_2}(\rho+3\beta_1+\beta_2)&=\rho+3\beta_1+\beta_2,
\end{aligned}$$ $\rho+\alpha$ isn't regular.

(b) 1) $\mathfrak{g}$ has type $\C_{\rk\mathfrak{g}}$, let be $\alpha=\beta_1+2(\beta_2+\cdots+\beta_{\rk\mathfrak{g}-1})+ \beta_{\rk\mathfrak{g}}$. For any positive root $\gamma$, $\<\alpha_1,\gamma\>$ is nonnegative, then $\alpha_1$ is dominant. Hence $\rho+\alpha_1$ is regular dominant. The other possible values of $\alpha$ are
$$\begin{aligned}
\alpha_{1,i,j}=\beta_i+\cdots+\beta_{j-1}, & 1\leqslant i<j\leqslant\rk\mathfrak{g}\\
\alpha_{2,i,j}=\beta_i+\cdots+\beta_{j-1}+2(\beta_j+\cdots+\beta_{\rk\mathfrak{g}-1})+\beta_{\rk\mathfrak{g}}, & 1<i<j\leqslant\rk\mathfrak{g}\\
\alpha_{3,j}=\beta_1+\cdots+\beta_j+2(\beta_{j+1} +\cdots+\beta_{\rk\mathfrak{g}-1})+ \beta_{\rk\mathfrak{g}}, & 1<j\leqslant\rk\mathfrak{g}-1\\
\alpha_{4,i} =2(\beta_i+\cdots+\beta_{\rk\mathfrak{g}-1})+\beta_{\rk\mathfrak{g}}, & 1<i\leqslant\rk\mathfrak{g}.
\end{aligned}$$
If $\alpha$ is equal to $\alpha_{1,i,j}$, $$s_{\beta_j}(\rho+\alpha)=\rho+\alpha,$$  if $\alpha$ is equal to $\alpha_{2,i,j}$, $$s_{\beta_{i-1}}(\rho+\alpha)=\rho+\alpha,$$ if $\alpha$ is equal to $\alpha_{3,j}$, $$s_{\beta_j}(\rho+\alpha)=\rho+\alpha,$$ if $\alpha$ is equal to $\alpha_{4,i}$, $$s_{\beta_{i-1}}(\rho+\alpha)=\rho+\alpha+\beta_{i-1}.$$In this case $\rho+\alpha$ is not regular since $\rho+\alpha+\beta_{i-1}=\rho+\alpha_{2,i-1,i}$. \\

2) $\mathfrak{g}$ has type $\D_{\rk\mathfrak{g}}$, let be $\alpha_1=\beta_1+2(\beta_2+\cdots+\beta_{\rk\mathfrak{g}-2})+ \beta_{\rk\mathfrak{g}-1}+\beta_{\rk\mathfrak{g}}$. For any positive root $\gamma$, $\<\alpha_1,\gamma\>$ is nonnegative, then $\alpha_1$ is dominant. Hence $\rho+\alpha_1$ is regular dominant. The other possible values of $\alpha$ are
$$\begin{aligned}
\alpha_{1,i,j}=\beta_i+\cdots+\beta_{j-1}, & 1\leqslant i<j\leqslant\rk\mathfrak{g}\\ \alpha_{2,i}=\beta_i+\cdots+\beta_{\rk\mathfrak{g}}, & 1\leqslant i<\rk\mathfrak{g}\\ \alpha_{3,i,j}=\beta_i+\cdots+\beta_{j-1}+2(\beta_j+\cdots+\beta_{\rk\mathfrak{g}-2})+ \beta_{\rk\mathfrak{g}-1}+ \beta_{\rk\mathfrak{g}}, & 1< i<j<\rk\mathfrak{g}-1\\
\alpha_{4,j}=\beta_1+\cdots+\beta_j+ 2(\beta_{j+1}+\cdots+\beta_{\rk\mathfrak{g}-2})+ \beta_{\rk\mathfrak{g}-1}+ \beta_{\rk\mathfrak{g}}, & 1<j<\rk\mathfrak{g}-2.
\end{aligned}$$
If $\alpha$ is equal to $\alpha_{1,i,j}$ for $i\notin\{1,\rk\mathfrak{g}-2,\rk\mathfrak{g}-1\}$, $$s_{\beta_{i-1}}(\rho+\alpha)=\rho+\alpha$$ for $i=1$ and $j\notin\{\rk\mathfrak{g}-2,\rk\mathfrak{g}-1,\rk\mathfrak{g}\}$, $$s_{\beta_j}(\rho+\alpha)=\rho+\alpha$$ for $i=1$ and $j=\rk\mathfrak{g}-2$, $$s_{\beta_{\rk\mathfrak{g}-2}}(\rho+\alpha)=\rho+\alpha$$ for $i=1$ and $j\in\{\rk\mathfrak{g}-1,\rk\mathfrak{g}\}$, $$s_{\beta_{\rk\mathfrak{g}}}(\rho+\alpha)=\rho+\alpha$$ for $i=\rk\mathfrak{g}-2$, $\alpha=\beta_{\rk\mathfrak{g}-2}+\beta_{\rk\mathfrak{g}-1}$ and $$s_{\beta_{\rk\mathfrak{g}}}(\rho+\alpha)=\rho+\alpha,$$ for $i=\rk\mathfrak{g}-1$, $\alpha=\beta_{\rk\mathfrak{g}-1}$ and $$s_{\beta_{\rk\mathfrak{g}-2}}(\rho+\alpha)=\rho+\alpha,$$  if $\alpha$ is equal to $\alpha_{3,i,j}$, $$s_{\beta_{i-1}}(\rho+\alpha)=\rho+\alpha,$$ if $\alpha$ is equal to $\alpha_{2,i}$ with $i=1$, $$s_{\beta_{\rk\mathfrak{g}-2}}(\rho+\alpha)=\rho+\alpha,$$ for $i\notin\{1,\rk\mathfrak{g}-1\}$, !
 $$s_{\beta_{i-1}}(\rho+\alpha)=\rho+\alpha,$$ for $i=\rk\mathfrak{g}-1$, $$s_{\beta_{\rk\mathfrak{g}-2}}(\rho+\alpha)=\rho+\beta_{\rk\mathfrak{g}-2}+ \beta_{\rk\mathfrak{g}-1}+\beta_{\rk\mathfrak{g}}$$ and since $\beta_{\rk\mathfrak{g}-2}+ \beta_{\rk\mathfrak{g}-1}+\beta_{\rk\mathfrak{g}}$ is equal to $\alpha_{2,i}$ with $i=\rk\mathfrak{g}-2$, $\rho+\beta_{\rk\mathfrak{g}-2}+ \beta_{\rk\mathfrak{g}-1}+\beta_{\rk\mathfrak{g}}$ is not regular,\\ if $\alpha$ is equal to $\alpha_{4,j}$, $$ s_{\beta_j}(\rho+\alpha)=\rho+\alpha.$$
\end{proof}

\end{document}